\documentclass[12pt]{article}
\usepackage[active]{srcltx}
\usepackage[curve]{xypic}
\usepackage{graphicx}

\usepackage[active]{srcltx}
\usepackage{epsfig,amsmath}
\usepackage[english]{babel}
\usepackage{amsmath,amsfonts,amssymb,amscd,color,epsfig,amsthm}

\newtheorem{newthm}{Theorem}
\newtheorem{theorem}{Theorem}[section]
\newtheorem{lemma}[theorem]{Lemma}
\newtheorem{proposition}[theorem]{Proposition}
\newtheorem{corollary}[theorem]{Corollary}

\newtheorem{definition}[theorem]{Definition}

\theoremstyle{remark}

\theoremstyle{plain}

\numberwithin{equation}{section}

\newcommand{\on}{\operatorname}

\def\CCC{{\cal C}}

\def\EEE{{\cal E}}

\def\CCC{{\cal C}}

\def\QQQ{{\cal Q}}

\def\VVV{{\cal V}}

\def\smm{\smallsetminus}

\def\C{\mbox{$\mathbb C$}}
\def\T{\mbox{$\mathbb T$}}

\def\D{\mathbb D}
\def\Z{\mbox{$\mathbb Z$}}
\def\Q{\QQQ}
\def\N{\mbox{$\mathbb N$}}

\def\lv{ \left(\begin{matrix} }
 \def\rv{\end{matrix}\right)}

\def\eqdef{:=}

\def\cal{\mathcal}

\def\dz{{\ \rm{d}z}}
\def\dw{{\dw}}

\def\ds{\displaystyle}

\newcommand{\mylabel}[1]{\label{#1}}

\newcommand{\REFEQN}[1] { \begin{equation}\mylabel{#1} }
\newcommand{\ENDEQN}{\end{equation}}
\newcommand{\REFTHM}[1] { \begin{theorem}\mylabel{#1} }
\newcommand{\ENDTHM}{\end{theorem}}
\newcommand{\REFNTH}[1] { \begin{newthm}\mylabel{#1} }
\newcommand{\ENDNTH}{\end{newthm}}
\newcommand{\REFPROP}[1]{\begin{proposition}\mylabel{#1} }
\newcommand{\ENDPROP}{\end{proposition} }
\newcommand{\REFLEM}[1]{\begin{lemma}\mylabel{#1} }
\newcommand{\ENDLEM}{\end{lemma} }
\newcommand{\REFCOR}[1]{\begin{corollary}\mylabel{#1} }
\newcommand{\ENDCOR}{\end{corollary} }

\def\smm{ {\smallsetminus }}
\def\ds{\displaystyle }

\def\mystrut{{\rule[-2ex]{0ex}{4.5ex}{}}}

\def\ov{\overline}

\def\T{{\mathbb T}}

\begin{document}

\title{\large Preperiodic dynatomic curves for $z\mapsto z^d+c$}
\author{Gao Yan}

\maketitle

\begin{abstract}
The preperiodic dynatomic curve $\mathcal{X}_{n,p}$ is the closure in $\C^2$ of the set consisting of $(c,z)$ such that $z$ is a preperiodic point of the polynomial $z\mapsto z^d+c$ with preperiod $n$ and period $p$ ($n,p\geq1$).
We prove that each $\mathcal{X}_{n,p}$ has  exactly $d-1$ irreducible components,   these components are all smooth and have pairwise transverse intersections  at the  singular points of $\mathcal{X}_{n,p}$. We also compute the genus of each component and  the Galois group of the defining polynomial of $\mathcal{X}_{n,p}$.

\vspace{0.1cm}
{\bf Keywords and phrases}: dynatomic curves, smooth, irreducible, Multibrot set, Galois group, complex dynamics.

\vspace{0.1cm}

{\bf AMS(2010) Subject Classification}: 14H50, 37F45.
\end{abstract}

\section{Introduction}

Fix $d\ge 2$. For $c\in \C$, set $f_c(z)=z^d+c$. For $p\geq 1$, define
 $$\check{\mathcal{X}}_{0,p}:=\bigl\{(c,z)\in \C^{2}\mid f^{p}_{c}(z)=z \text{ and for all } 0<k< p,\ \ f^{k}_c(z)\ne z\bigr\}.$$
$$\mathcal{X}_{0,p}:=\text{the closure of $\check{\mathcal{X}}_{0,p}$ in $\C^2$ }.$$

It is  known that all $\mathcal{X}_{0,p}$ are affine algebraic curves, called the {\it periodic\ dynatomic curves}. These curves have been the subject of several studies in algebraic and holomorphic
dynamical systems. The known results for these curves mainly include the smoothness (Douady-Hubbard \cite{DH1}, Milnor \cite{Mil1}, Buff-Tan \cite{BT}); irreducibility (Bousch \cite{B}, Buff-Tan \cite{BT}, Morton \cite{Mo}, Lau-Schleicher \cite{LS}, Schleicher \cite{S}); the genus (Bousch \cite{B}) and the associated Galois groups (Bousch \cite{B}, Morton \cite{Mo}, Lau-Schleicher \cite{LS}, Schleicher \cite{S}).

In the present work, we study  some topological and algebraic properties of {\it preperiodic\ dynatomic\ curves}.

\begin{definition}
For $n\geq0,\ p\geq1$, a point $z$ is called a  $p$-periodic point if $f^p_c(z)=z$ but $f^k_c(z)\not=z$ for $0<k<p$, and an $(n,p)$-preperiodic point of $f_c$ if $f_c^n(z)$ is a $p$-periodic point of $f_c$ but $f_c^l(z)$ is not periodic for any $0\leq l<n$.
\end{definition}

Now, for any $n\geq1, p\geq 1$, define
$$\check{\mathcal{X}}_{n,p}=\bigl\{(c,z)\in\C^2\big |z \text{ is a }(n,p)\text{-preperiodic point of }f_c \bigr\}$$
$$\mathcal{X}_{n,p}:=\text{ the closure of $\check{\mathcal{X}}_{n,p}$ in $\C^2$}\ .$$

 In fact, as we shall see below, all $\mathcal{X}_{n,p}$ are also  affine algebraic curves,  called the $preperiodic\ dynatomic\ curves$.
 Limited work has been done for this kind of curves. The special case $d=2$ has been previously studied by Bousch \cite{B}, who
established  in this case that for any integers $n,p\geq1$, the curve $\mathcal{X}_{n,p}$ is also smooth and irreducible  (as the periodic dynatomic curves), and computed its  associated Galois group.

The main purpose of this work is to extend these results to arbitrary $d\geq2$. An obvious difference with the previous
case is that, for $d>2$, the curve $\mathcal{X}_{n,p}$ is no longer irreducible: it consists of $d-1$ irreducible
components. We may understand this by a simple observation.  Consider the curve $\mathcal{X}_{1,p}$ of $(1,p)$-preperiodic points, that
is, the points $z$ which are not p-periodic, but whose image $z_0=f(z)$ is. The periodic point
$z_{p-1}=f^{p-1}(z_0)$ is another preimage of $z_0$. Because $f_c(z)=z^d+c$, we have $z=\omega z_{p-1}$, where  $\omega$ is a $d$-th root of unity. According to the value of $\omega$, we
can partition the $(1,p)$-preperiodic points into $d-1$ classes, and this decomposition is of
algebraic nature: it corresponds to a factorization of $f^{p+1}_c(z)-f_c(z)$.

We  show that these $d-1$ components are smooth and irreducible.  Our approach to smoothness is by using elementary calculations on quadratic differentials and Thurston's  contraction principle, following the method of Buff-Tan (\cite{BT}). The approach to irreducibility is based on the connectedness of periodic dynatomic curves and then by an induction on the preperiodic index $n$.  Moreover, we  study the features of the singular points of $\mathcal{X}_{n,p}$.

 Following  Bousch, we compute the genus of each irreducible component and the associated Galois group of the curve $\mathcal{X}_{n,p}$.

Here is a list our main results. They are to be compared with results of periodic dynatomic curves.

Denote by  $\{\nu_d(p)\}_{p\geq 1}$ the unique sequence of positive integers satisfying the recursive relation
 \begin{equation}\label{nu}
  d^p=\sum_{k|p} \nu_d(k),\quad \text{integer }d\geq 2
  \end{equation}
and let $\varphi(m)$ be the Euler totient function (i.e. the number of positive integers less than $m$ and co-prime to $m$).
 For $n,p\geq1$, define the numbers $$M_{n,p}:=\nu_d(p)d^{n-2}(d-1)\big(n-1-\sum_{t=1}^{[\frac{n-1}{p}]}d^{-tp}\big),$$ where $[x]$ denotes the maximal integer less than or equal to $x$, and
\[K_{n,p}:=\nu_d(p)(d^{p-1}-1)d^{n-1-p}\big(\sum_{t=1}^{[\frac{n-1}{p}]-1} d^{-t(p-1)}-\sum_{t=1}^{[\frac{n-1}{p}]-1}d^{-pt}\big)+(d^{[\frac{n-1}{p}]}-1)\nu_d(p)d^{n-2-[\frac{n-1}{p}]p}\]
(one can refer to (\ref{number1}) and (\ref{number2}) for the computation of them).
For $n,p\geq 1$, set
$$g_p(d)= 1+\dfrac{dp-d-p-1}{2d} \nu_d(p)-\dfrac{d-1}{2d} \sum_{k|p, k<p} \varphi\Big(\dfrac p k\Big) k\cdot \nu_d(k),$$
\[g_{n,p}=1+\dfrac{1}{2}\nu_d(p)d^{n-2}(pd-d-p-1)+\dfrac{1}{2}(M_{n,p}+K_{n,p})-\frac{1}{2}d^{n-2}(d-1)\sum_{k|p,k<p} \varphi\Big(\dfrac p k\Big) k\cdot \nu_d(k).\]

\begin{theorem}
For any $d\geq 2,\ n,p\geq1$,  the preperiodic dynatomic curve $\mathcal{X}_{n,p}$ has the following properties :
\begin{enumerate}
\item The set $\mathcal{X}_{n,p}$ is an affine algebraic curve. It has $d-1$ irreducible components and each one is smooth. Moreover,  these components  are pairwise intersecting
at the singular points of $\mathcal{X}_{n,p}$. In particular, if $d=2$, the curve $\mathcal{X}_{n,p}$ is smooth and irreducible.
\item The genus of every irreducible component of $\mathcal{X}_{n,p}$ (in some kind of compactification) is $g_{n,p}(d)$,
  and all irreducible components are mutually homeomorphic.
\item The Galois group associated with $\mathcal{X}_{n,p}$ is the same as that associated with $\mathcal{X}_{\leq n,p}:=\cup_{l=0}^n\mathcal{X}_{l,p}$, which consists of all permutations on the roots of the defining polynomial of $\mathcal{X}_{\leq n,p}$ that commute with $f_c$ and the rotation of argument $1/d$.
\end{enumerate}
\end{theorem}

Here is a tableau comparing these various curves,  where $\mathbf{S}_m$ denotes the group of permutations on $\{1,\ldots,m\}$ and $G_{n,p}(d)$ is the associated Galois group of $\mathcal{X}_{n,p}$.

\begin{center}
\begin{tabular}{|c|| c | c |c}
\hline
periodic $\mathcal{X}_{0,p}$ & $d=2 $ & $d> 2$\mystrut \\ \hline \hline
& irreducible & irreducible \mystrut \\ \hline
& smooth & smooth \mystrut \\ \hline
genus & $g_p(2)$ &   $g_p(d)$  \mystrut \\ \hline
Galois group \mystrut &$ \mathbf{S}_{\nu_2(p)/p}\ltimes\Z_p^{\nu_2(p)/p}$&$ \mathbf{S}_{\nu_d(p)/p}\ltimes \Z_p^{\nu_d(p)/p}$\\
\hline
\end{tabular}
\end{center}
\vspace{0.5cm}

\begin{tabular}{|c|| c | c|c }
\hline
preperiodic $\mathcal{X}_{n,p}$, $n\ge 1$ & $d=2 $ & $d> 2$\mystrut \\ \hline \hline
& irreducible & $d-1$ irreducible components \mystrut \\ \hline
& smooth & not smooth, but each component is smooth \mystrut \\ \hline
component-wise genus &  $g_{n,p}(2)$ &   $g_{n,p}(d)$  \mystrut \\ \hline
Galois group\mystrut  &$G_{n,p}(2)$ &$G_{n,p}(d)$  \mystrut \\ \hline
pairwise intersection &empty&$C_{n,p}(\text{singular}):$ singularity set of $\mathcal{X}_{n,p}$ \mystrut \\
\hline
\end{tabular}

\vspace{0.2cm}

This manuscript is organized as follows:

In section $2$, we summarize some preliminaries  that will be used in this paper.

In section $3$, we will prove that every $\mathcal{X}_{n,p}$ is an affine algebraic curve and find its defining polynomial.

In section $4$, we give the irreducible factorization of $\mathcal{X}_{n,p}$, and prove that each irreducible factor is smooth and these
irreducible components are pairwise  intersecting at the singular points of $\mathcal{X}_{n,p}$.

In section $5$, we  calculate the genus of each irreducible component.

In section $6$, we  describe $\mathcal{X}_{n,p}$ from the algebraic point of view  by calculating its associated Galois group.

Acknowledgement. I thank Tan Lei for helpful discussions and suggestions.

\section{Preliminaries}
{\bf 1. Filled in Julia set and Multibrot set.}
These material can be found in \cite{DH1,DH2} and \cite{DE}.

For $c\in \C$, we denote by $K_c$ the filled-in Julia set of $f_c$, that is the set of points $z\in \C$ whose orbit under $f_c$ is bounded. We denote by $M_{d}$ the {\it Multibrot\ set}
in the parameter plane, that is the set of parameters $c\in \C$ for which the critical point $0$ belongs to $K_c$. It is known that $M_d$ is connected.

Assume $c\in M_{d}$. Then $K_c$ is connected. There is a conformal isomorphism $\phi_c:\C\smm \overline K_c\to \C\smm \overline \D$ satisfying $\phi_c\circ f_c=  \big(\phi_c\big)^{d}$ and $\phi_c^\prime(\infty)=1$ (i,e. $\dfrac{\phi_c(z)}{z}\longrightarrow_{z\to \infty} 1 $). The dynamical ray of angle $\theta\in \T$ is defined by
\[R_c(\theta):=\bigl\{z\in \C\smm K_c\mid \arg\bigl(\phi_c(z)\bigr)=2\pi\theta\bigr\}.\]

Assume $c\notin M_{d}$. Then $K_c$ is a Cantor set and  all periodic points of $f_c$ are {\it repelling}, that is $|(f^p)^\prime(z)|>1$ for $p\geq1$ and all $p$-periodic point $z$. There is a conformal isomorphism $\phi_c:U_c\to V_c$ between neighborhoods of $\infty$ in $\C$, which satisfies  $\phi_c\circ f_c= \big( \phi_c\big)^{d}$ on $U_c$. We may choose $U_c$ so that $U_c$ contains the critical value $c$ and  $V_c$ is the complement of a closed disk.  For  each $\theta\in \T$,  there is an infimum $r_c(\theta)\geq 1$ such that $\phi_c^{-1}$ extends analytically along $R_0(\theta)\cap \bigl\{z\in \C\mid r_c(\theta)<|z|\bigr\}$. We denote by $\psi_c$ this extension and by $R_c(\theta)$ the dynamical ray
\[R_c(\theta):=\psi_c\Big(R_0(\theta)\cap \bigl\{z\in \C\mid r_c(\theta)<|z|\bigr\}\Big).\]
As $|z|\searrow r_c(\theta)$, the point $\psi_c(r{\rm e}^{2\pi i\theta})$ converges to a point $x\in \C$ (\cite[Pro.8.3]{DH2}). If $r_c(\theta)>1$, then $x\in \C\smm K_c$ is an iterated preimage of $0$ and we say that $R_c(\theta)$ bifurcates at $x$. If $r_c(\theta)=1$, then $x$ belongs to $K_c$ and we say that $R_c(\theta)$ lands at $x$.

There are three kinds of important parameters in $M_d$: super-attracting, parabolic, and Misiurewicz parameters. Recall that a point $z$ is said to be $p$-periodic if $f^p_c(z)=z$ but $f^k_c(z)\not=z$ for $0<k<p$. We call $c\in\C$
\begin{itemize}
\item a {\it $p$-super-attracting parameter} if $0$ is $p$-periodic by $f_c$;
\item a {\it $p$-parabolic parameter} if $f_c$ has a $p$-periodic point $z_0$ with $(f^p)^\prime(z_0)=1$ or a $m$-periodic point $z_0$ such that $m\mid p$ and $(f^m)^\prime(z_0)$ is a $\dfrac{p}{m}$-th root of unity;
\item a $(n,p)$-{\it Misiurewicz parameter} if $0$ is a $(n,p)$-preperiodic point of $f_c$.
\end{itemize}

A well-known result in complex dynamics says that any parabolic cycle of a rational map has a critical point in its basin, whose orbit eventually converges to , but is disjoint with the cycle (see \cite[Thm.10.15]{Mil2}).
So for the family of  unicritical polynomials $\{f_c\mid c\in\C\}$, the three classes of parameters above are pairwise disjoint. We write this point as a lemma, since it will  be repeatedly used throughout the paper.
\begin{lemma}\label{specificity}
If the critical point $0$ is (pre)periodic for $f_c$, then $c$ is not a parabolic parameter.
\end{lemma}

{\noindent\bf 2. Affine algebraic curve and singularity.} These material can be found in \cite{G}.

A polynomial $f\in\C[x,y]$ is called {\it squarefree} if it is not divisible by $h(x,y)^2$ for any non-constant $h(x,y)\in\C[x,y]$.
An {\it affine algebraic curve over} $\C$ is defined as
\[\CCC=\{(x,y)\in\C^2\mid f(x,y)=0\}\]
where $f$ is a non-constant squarefree polynomial in $\C[x,y]$, called the {\it defining polynomial} of $\mathcal{C}$.
If $f=\prod_{i=1}^m f_i$, where $f_i$ are the irreducible factor of $f$,
we say that the affine curve defined by $f_i$ is a {\it irreducible component} of $\mathcal{C}$.

Let $f\in\C[x,y]$. The {\it total degree} of $f(x,y)$  as a multivariate polynomial is the highest degree of its terms, denoted by Deg$(f)$. Correspondingly, we denote by deg$_x(f)$ and deg$_y(f)$ the degrees of $f$ when considered as a polynomial in the variable $x$ and $y$ respectively. The following lemma is repeatedly used in this paper.
\begin{lemma}\label{degree}
\begin{description}
\item [(1)] If $f=f_1f_2$ with $f_1,f_2\in \C[x,y]$, then Deg$(f)=$ Deg$(f_1)+$Deg$(f_2)$, deg$_x(f)=$ deg$_x(f_1)+$deg$_x(f_2)$ and deg$_y(f)=$
deg$_y(f_1)+$deg$_y(f_2)$.
\item [(2)] For $f_1,f_2\in\C[x,y]$, if $f(x,y)=f_1(x,f_2(x,y))$, then $\text{deg}_y(f)=\text{deg}_y(f_1)\cdot \text{deg}_y(f_2)$.
\item [(3)] For $f_1,f_2\in\C[x,y]$, if $f(x,y)=f_1(x,f_2(x,y))$ and $\text{Deg}(f_1)=\text{deg}_y(f_1)\geq1$, $\text{Deg}(f_2)>1$, then $\text{Deg}(f)=\text{Deg}(f_1)\cdot \text{Deg}(f_2)$.
\end{description}
\end{lemma}
\begin{proof}
(1). Refer to \cite[section 1.1]{F}.

\noindent (2). It is straightforward by a simple computation.

\noindent (3). Set $d_1:=\text{Deg}(f_1)$ and $d_2:=\text{Deg}(f_2)$. According to the conditions of the lemma, we have $\text{deg}_y(f_1)=d_1\geq1$ and $d_2>1$. On one hand, since $\text{Deg}(f_1)=\text{deg}_y(f_1)=d_1$, then there is a unique term in $f_1$ with the form $a_1y^{d_1}$, where $a_1$ is a non-zero constant. So, by (1) and $d_1\geq1$, it follows that
$\text{Deg}(a_1f_2^{d_1})=d_1d_2$. On the other hand, any other term of $f_1$ has the form  $ax^sy^t$, where $a$ is a non-zero constant and either $s+t<d_1$ or $s+t=d_1$ and $s\geq1$. According to point (1) and $d_2>1$,
\[\text{Deg}(x^sf_2^t(x,y))=s+td_2<  d_1d_2.\]
 So we get $\text{Deg}(f)=d_1d_2$.
\end{proof}

Let $\mathcal{C}$ be an affine algebraic curve for $\C$ defined by $f\in\C[x,y]$, and let $P=(a,b)\in\mathcal{C}$. The multiplicity of $\mathcal{C}$ at $P$, denoted by  $mult_P(\mathcal{C})$, is defined as the order $s$ of the first non-vanishing term in the Taylor expansion of $f$ at $P$, i.e.
\[f(x,y)=\sum_{s=0}^{\infty}\frac{1}{s!}\sum_{t=0}^{s}{s\choose t}(x-a)^t(y-b)^{s-t}\frac{\partial^sf}{\partial x^t\partial y^{s-t}}(a,b).\]
If $mult_P(\mathcal{C})=1$, the point $P$ is called a {\it smooth point} of $\mathcal{C}$. If  $mult_P(\mathcal{C})=r>1$, then we say that $P$ is a {\it singular point of multiplicity $r$}. We say that $\mathcal{C}$ or
$f$ is $smooth$ if any point on $\mathcal{C}$ is smooth. Note that the first non-vanishing term is a homogeneous polynomial about $x-a$ and $y-b$, so all its irreducible factors are linear and they are called the $tangents$ of $\mathcal{C}$ at $P$.

A singular point $P$ of multiplicity $r$ on an affine plane curve $\mathcal{C}$ is called {\it ordinary} if the $r$ tangents to $\mathcal{C}$ at $P$ are distinct.

The following result provides a topological interpretation of the irreducibility of polynomials.
\begin{lemma}\label{relation}
A squarefree polynomial $f\in\C[x,y]$ is irreducible if and only if the set of  smooth points of $f$ is connected.
\end{lemma}

{\noindent \bf 3. Periodic dynatomic curves.}

In this paper, some of proofs and statements rely on  the work of the periodic curves $\mathcal{X}_{0,p}$. We list the related  results  in the following lemma.  Its proof  can be found in \cite{B}, \cite{BT}, \cite{DE}, \cite{GO}, \cite{LS}, \cite{Mil1}, \cite{S}.

By abuse of notation, we will identify  polynomials in
$\C[c,z]$ as  polynomials in $\mathbf{C}[z]$ with $\mathbf{C}=\C[c]$. Denote by $\mathbf{K}$ a fixed algebraically closed field containing $\mathbf{C}$.

Let $f\in\C[c,z]$.  By the {\it zeros of} $f\in\C[c,z]$, we mean the points $(c,z)\in\C^2$ with $f(c,z)=0$. By the {\it roots of $f\in\mathbf{C}[z]$}, we mean the roots of $f$ in $\mathbf{K}$ when it is considered as a polynomial in $\mathbf{C}[z]$.

Recall that $\{\nu_d(p)\}_{p\geq1}$ is a unique sequence of positive integers satisfying the recursive relation $d^p=\sum_{k|p}\nu_d(k)$, $\text{Deg}(f)$ denotes the total degree of $f$ and $\text{deg}_z(f)$ denotes the degree of $f$ as a polynomial in $\mathbf{C}[z]$.

\begin{lemma} \label{period}

Let $\mathcal{X}_{0,p}$ be a periodic dynatomic curve. Then

\begin{enumerate}
\item \cite{B,BT} There exists a unique sequence of monic polynomials $\{Q_{0,p}\in\mathbf{C}[z]\}_{p\geq1}$ such that  for all $p\geq1$,
\[\Phi_{0,p}(c,z):=f^{\circ p}_c(z)-z=\prod_{k|p}Q_{0,k}(c,z).\]
Moreover, we have $\text{Deg}(Q_{0,p})=\text{deg}_z(Q_{0,p})=\nu_d(p)$.

\item \cite{BT} Let $c_0$ be an arbitrary parameter. Then a point $z_0$ is a root of $Q_{0,p}(c_0,z)\in\C[z]$ if and only if one of the three exclusive conditions is satisfied:\label{12}
\begin{description}
\item[(1)] $z_0$ is a $p$-periodic point of $f_{c_0}$  and $[f_{c_0}^{\circ p}]^\prime(z_0)\ne1$,
\item[(2)] $z_0$ is a $p$-periodic point of $f_{c_0}$ and $[f_{c_0}^{\circ p}]^\prime(z_0)=1$,
\item[(3)] $z_0$ is an $m$-periodic point of $f_{c_0}$, where $m$ is a proper factor of $p$, and $[f_{c_0}^{\circ m}]^\prime(z_0)$ is a primitive $\frac{p}{m}$-th  root of unity\label{(3)}.
\end{description}
\item \cite{B,BT,GO,LS,S} The polynomial $Q_{0,p}$ is smooth and irreducible for all $p\geq 1$ and
\[\mathcal{X}_{0,p}=\{(c,z)\in\C\ |\ Q_{0,p}(c,z)=0\}.\]
\item \cite{B,BT,GO} The projection $\pi_{0,p}: \mathcal{X}_{0,p}\longrightarrow\C$, defined by $\pi_{0,p}(c,z)=c$, is a degree $\nu_d(p)$ (given in (\ref{nu})) branched covering with two kinds of critical points:

\begin{description}
\item[(1)] $C_{0,p}\text{(primitive)}=\{(c,z)\in\mathcal{X}_{0,p}\mid(c,z)\text{ satisfies condition }(2) \text{ of point } \ref{12}\}$. In this case,  $(c,z)$ is a simple critical point.
\item[(2)] $C_{0,p}\text{(satellite)}=\{(c,z)\in\mathcal{X}_{0,p}\mid(c,z)\text{ satisfies condition }(3) \text{ of point } \ref{12}\}$. In this case,
the multiplicity of the critical point $(c,z)$ is $\frac{p}{m}-1$.
\end{description}
The critical value set of $\pi_{0,p}$   consists of the parabolic parameters of period $p$.
\item \cite{DE,Mil1} The projection $\varpi_{0,p}:\mathcal{X}_{0,p}\longrightarrow \C$, defined by $\varpi_{0,p}(c,z)=z$, is a degree $\nu_d(p)/d$ branched covering, which is injective near each point $(c_0,0)\in\mathcal{X}_{0,p}$.
\item \cite{B} The Galois group $G_{0,p}$ for the polynomial $Q_{0,p}\in \mathbf{C}[z]$ consists of  the permutations on roots of $Q_{0,p}\in\mathbf{C}[z]$ that  commute with  $f_c$.
\end{enumerate}

\end{lemma}

\section{The defining polynomial for $ \mathcal{X}_{n,p}$}\label{3}

The objective of this section is to show that $\mathcal{X}_{n,p}$  is an affine algebraic curve and find its defining polynomial.

Recall that $\mathbf{C}$ denotes the ring $\C[c]$. For $n\geq0,p\geq1$, set $\Phi_{n,p}(c,z)=f_c^{\circ (n+p)}(z)-f_c^{\circ n}(z)$.
\begin{lemma}\label{squarefree}
The polynomial $\Phi_{n,p}\in\mathbf{C}[z]$ has no repeated roots. Consequently, it is squarefree.
\end{lemma}
\begin{proof}
To prove this lemma, it is enough to show that there exists $c_0\in\C$ such that all roots of $\Phi_{n,p}(c_0,z)$ are simple. In fact, given $c_0\in\C\setminus M_d$, a point $z_0$ is a root of $\Phi_{n,p}(c_0,z)\in\mathbb{C}[z]$ if and only if $z_0$ is a $(l,k)$-preperiodic point of $f_{c_0}$, where $0\leq l\leq n$ and $ k\mid p$. For such a $c_0$, the critical point $0$ goes to infinity and all periodic points of $f_{c_0}$ are repelling. It follows that
\[(\partial \Phi_{n,p}/\partial z)(c_0,z_0)=[f_{c_0}^{\circ n}]^\prime(z_0)\bigl([f_{c_0}^{\circ p}]^\prime(z_0)-1\bigr)\ne0,\]
which completes the proof.
\end{proof}

\begin{lemma} \label{polynomial}
There exists a unique double indexed sequence of  squarefree, monic polynomials $\{Q_{n,p}\in\mathbf{C}[z]\}_{n,p\geq1}$, such that  for all $n,p\geq 1$,
\begin{equation}\label{e1}\Phi_{n,p}(c,z)=\Phi_{n-1,p}(c,z)\prod_{k|p}Q_{n,k}(c,z).\end{equation}
Moreover, we have $\text{Deg}(Q_{n,p})=\text{deg}_z(Q_{n,p})=\nu_d(p)(d-1)d^{n-1}$.
\end{lemma}

\begin{proof}
The definition of $\{Q_{n,p}\}_{n,p\geq1}$ is based on the polynomials $\{Q_{0,p}\}_{p\geq1}$ which appear in part 1 of Lemma \ref{period}. We firstly show that $Q_{0,p}(c,z)$ divides $Q_{0,p}(c,f_c(z))$ for any $p\geq1$.
Since the polynomials $Q_{0,p}(c,f_c(z))\in \mathbf{C}[z]$ are monic, we may perform a Euclidean division to find  a monic quotient $Q\in\mathbf{C}[z]$ and a remainder $R\in\mathbf{C}[z]$ with degree$(R)<$degree$(Q_{0,p})$, such that $Q_{0,p}(c,f_c(z))=Q_{0,p}Q+R$. We need to show that $R=0$, which enable us to set $Q_{1,p}(c,z):=Q$.

Following Lemma \ref{squarefree} and part 1 of Lemma \ref{period}, the polynomial $Q_{0,p}\in\mathbf{C}[z]$ does not have repeated factors. So its discriminant $\Delta_{0,p}\in \mathbb{C}[c]$ does not identically vanish, and hence $\Delta_{0,p}(c)\not=0$ outside a finite set.
Fix $c_0\in\C$ such that $\Delta_{0,p}(c_0)\not=0$. Then any root $z_0$ of $Q_{0,p}(c_0,z)$ is simple. By part 2 of Lemma \ref{period}, the point $z_0$ is  also a root of $Q_{0,p}(c_0,f_{c_0}(z))$. As a consequence, $R(c_0,z)=0$ for all $z\in\C$. Since this is true for every $c_0$ outside a finite set, we have $R=0$ as required.

For $n,p\geq1$, we define $Q_{n,p}(c,z):=Q_{1,p}(c,f^{n-1}_c(z))$. It is clear that each $Q_{n,p}\in\mathbf{C}[z]$ is monic. Note that   $\Phi_{n,p}(c,z)=\Phi_{0,p}(c,f^n_c(z))$ for any $n,p\geq1$, then we have
\begin{eqnarray*}
\Phi_{n,p}(c,z)&=&\Phi_{0,p}(c,f^n_c(z))\overset{Lem.\ref{period}}{=}\prod_{k|p}Q_{0,k}(c,f^n_c(z))=\prod_{k|p}Q_{0,k}(c,f_c^{n-1}(z))Q_{1,k}(c,f^{n-1}_c(z))\\
&=&\prod_{k|p}Q_{0,k}(c,f^{n-1}_c(z))\prod_{k|p}Q_{1,k}(c,f^{n-1}_c(z))=\Phi_{0,p}(c,f^{n-1}_c(z))\prod_{k|p}Q_{n,k}(c,z)\\
&=&\Phi_{n-1,p}(c,z)\prod_{k|p}Q_{n,k}(c,z).
\end{eqnarray*}
Since each $\Phi_{n,p}$ is squarefree (Lemma \ref{squarefree}), so is each $Q_{n,p}$.

Repeatedly applying (2) and (3) of Lemma \ref{degree}, we have $\text{Deg}(f^k_c(z))=\text{deg}_z(f^k_c(z))=d^k$ for $k\geq1$. It follows that
$\text{Deg}(\Phi_{n,p})=\text{deg}_z(\Phi_{n,p})=d^{n+p}$ for $n\geq0,p\geq 1$.  Then by the recursive formulas (\ref{e1}), (\ref{nu}) and  (1) of Lemma \ref{degree}, the degree conclusion in the lemma holds.
\end{proof}

By the definition of $Q_{n,p}$,  we get the inductive formulas
\begin{equation}\label{induction}
\left\{
  \begin{array}{ll}
    Q_{n-1,p}(c,f_c(z))=Q_{n,p}(c,z), & n\geq2;\\
    Q_{0,p}(c,f_c(z))=Q_{0,p}(c,z)Q_{1,p}(c,z), & n=1.
  \end{array}
\right.
\end{equation}
for each $p\geq1$. This equation implies that we can obtain the properties of $Q_{n,p}$
by induction on $n$.

In fact, $Q_{n,p}(c,z)$ is the defining polynomial of $\mathcal{X}_{n,p}$. To see this, we will now study the
properties of the roots of $Q_{n,p}(c_0,z)\in\C[z]$ for an arbitrary parameter $c_0\in\C$.
\begin{proposition}\label{define}
Let $n\geq1,\ p\geq 1$ be any pair of integers and $c_0\in\C$ be any parameter.
Then $z_0\in\C$ is a root of $Q_{n,p}(c_0,z)$ if and only if one of  the following $5$ mutually exclusive conditions holds:
\begin{description}
\item[(1)] $z_0$ is a $(n,p)$-preperiodic point of $f_{c_0}$ such that  $f_{c_0}^l(z_0)\ne0$ for any $0\leq l<n$ and $[f^{ p}_{c_0}]^\prime(f^{ n}_{c_0}(z_0))\ne1.$
\item[(2)] $z_0$ is a $(n,p)$-preperiodic point of $f_{c_0}$ such that  $f_{c_0}^l(z_0)\ne0$ for any $0\leq l<n$ and $[f^{ p}_{c_0}]^\prime(f^{ n}_{c_0}(z_0))=1.$
\item[(3)] $z_0$ is a $(n,m)$-preperiodic point of $f_{c_0}$ such that  $f_{c_0}^l(z_0)\ne0$ for any $0\leq l<n$ and $m$  is a proper factor of $p$ with $[f^{m}_{c_0}]^\prime(f^{ n}_{c_0}(z_0))$  a primitive $\frac{p}{m}$-th root of unity.
\item[(4)] $z_0$ is a $(n,p)$-preperiodic point  of $f_{c_0}$  such that $f_{c_0}^l(z_0)=0$ for some $0\leq l<n.$

\item [(5)] $f_{c_0}^{(n-1)}(z_0)=0$ and $0$ is a $p$-periodic point of $f_{c_0}$.
\end{description}
\end{proposition}
We remark that in case (4), the case of $l=n-1$ never occurs.
\begin{proof}
Fix $c_0\in\C$. The proof goes by induction on $n$.
As $n=1,\ Q_{0,p}(c,f_c(z))=Q_{0,p}(c,z)\cdot Q_{1,p}(c,z)$. We claim that $z_0$ is a common root of $Q_{0,p}(c_0,z)$ and $Q_{1,p}(c_0,z)$ if and only if $z_0=0$ is a $p$ periodic point of $f_{c_0}$.

For sufficiency, we only need to note that, in this case, $0$ is a multiple root of $Q_{0,p}(c_0,f_{c_0}(z))$, but a simple root of $Q_{0,p}(c_0,z)$ by part 2 of Lemma \ref{period}. For necessity, $z_0$ must be a multiple root of $Q_{0,p}(c_0,f_{c_0}(z))$. It follows that  either $f_{c_0}(z_0)$ is a multiple root of $Q_{0,p}(c_0,z)$ or $z_0$ is a critical point of $f_{c_0}$.
In the former case, by $4$ of Lemma \ref{period}, $c_0$ is a parabolic parameter and $f_{c_0}(z_0)$ is a parabolic periodic point.  It means that $Q_{0,p}(c_0,f_{c_0}(z))$ and $Q_{0,p}(c_0,z)$ have the same zero multiplicity at $z_0$. Then $Q_{1,p}(c_0,z_0)\not=0$. In the latter case, we have $z_0=0$, and by 2 of Lemma \ref{period} $0$ is a $p$ periodic point of $f_{c_0}$.

Such $c_0,z_0$ correspond to the condition $(5)$.
In any other case,  $z_0$ is a root of $Q_{1,p}(c_0,z)$ if and only if $f_{c_0}(z_0)$ is a root of $Q_{0,p}(c_0,z)$
but $z_0$ is not periodic. In fact, if $z_0$ were periodic, it would have the same period and multiplier as its first image.  By part 2 of Lemma \ref{period}, we get that $Q_{0,p}(c_0,z_0)$ vanishes, which leads to a contradiction. Then part $2$ of Theorem \ref{period} implies that $z_0$ satisfies one of the conditions $(1),(2),(3),(4)$ in Proposition \ref{define}.

 Assume that the proposition is established for $1\leq l<n$. At this time, $Q_{n,p}(c,z)=Q_{n-1,p}(c,f_c(z))$. So for any $c_0\in\C$, $z_0$ is a root of $Q_{n,p}(c_0,z)$ if and only if $f_{c_0}(z_0)$ is a root of $Q_{n-1,p}(c_0,z)$. By Lemma \ref{specificity}, if $f_{c_0}(z_0)$ satisfies property (2) or (3), then the orbit of $z_0$ does not contain $0$. Therefore by the inductive assumption, the point  $z_0$  satisfies one of the $5$ exclusive conditions  in Proposition \ref{define}.
\end{proof}
In Proposition \ref{define}, the zeros of $Q_{n,p}(c,z)$ are divided into $5$ classes. We give some notations to denote the sets consisting of zeros of most classes in the following table.

\begin{tabular}{|l|l|}
\hline
The set&\qquad The points in the set \\\hline
$C_{n,p}$(primitive)& $(c,z)\text{ satisfies the condition }(2)\text{ in Proposition }\ref{define}$\\\hline
$C_{n,p}$(satellite)& $(c,z)\text{ satisfies the condition }(3)\text{ in Proposition }\ref{define}$\\\hline
$C_{n,p}$(Misiurewicz)& $(c,z)\text{ satisfies the condition }(4)\text{ in Proposition }\ref{define}$\\\hline
$C_{n,p}$(singular)& $(c,z)\text{ satisfies the condition }(5)\text{ in Proposition }\ref{define}$\\\hline
\end{tabular}

{Recall that for any $n, p\geq 1$, {the sets} $\check{\mathcal{X}}_{n,p}$ and $\mathcal{X}_{n,p}$ are defined by
$$\check{\mathcal{X}}_{n,p}=\bigl\{(c,z)\in\C^2\big |z \text{ is a }(n,p)\text{-preperiodic point of }f_c \bigr\}$$
and
$$\mathcal{X}_{n,p}:=\text{ the closure of $\check{\mathcal{X}}_{n,p}$ in $\C^2$}\ .$$}

 \begin{proposition}
 For $n,p\geq1$, we have
\[\mathcal{X}_{n,p}=\bigl\{(c,z)\big| Q_{n,p}(c,z)=0\bigr\}\quad\text{and}\quad\mathcal{X}_{n,p}\setminus\check{\mathcal{X}}_{n,p}=C_{n,p}(\text{satellite})\cup C_{n,p}(\text{singular})\]
\end{proposition}
\begin{proof}
Set $X:=\{(c,z)\mid Q_{n,p}(c,z)=0\}$. Then $X$ is a closed, perfect set. By the definition of $\check{\mathcal{X}}_{n,p}$ and Proposition \ref{define}, we have
\begin{equation}\label{closure}
X\setminus\big(\ C_{n,p}(\text{satellite})\cup C_{n,p}(\text{singular})\ \big)\ = \check{\mathcal{X}}_{n,p}\subset X
\end{equation}
We claim that the sets $C_{n,p}(\text{satellite})$ and $C_{n,p}(\text{singular})$ are both finite. If so, we get
\[X=\ov{X\setminus(C_{n,p}(\text{satellite})\cup C_{n,p}(\text{singular}))}\ =\ \ov{\check{\mathcal{X}}_{n,p}}=\mathcal{X}_{n,p}\subset X.\]
Hence it remains to check the claim.

If $(c_0,z_0)\in C_{n,p}(\text{satellite})$, it satisfies that $f^{n+p}_{c_0}(z_0)-f^n_{c_0}(z_0)=0$ and $[f^{ p}_{c_0}]^\prime(f^{ n}_{c_0}(z_0))=1$. Hence $c_0$ is a root of the resultant $R\in\mathbb{C}[c]$ of the equations $f^{n+p}_c(z)-f^n_c(z)=0$ and $(f_c^p)'(f^n_c(z))=1$. For a parameter $c$ outside the Multibrot set, all the periodic points of $f_{c}$ are repelling, so the polynomials  $f^{n+p}_{c}(z)-f^n_{c}(z)$ and $(f^{ p}_{c})^\prime(f^{ n}_{c}(z))-1$ do not have a common root. It follows that $R$ is not identically zero, and hence, its roots form a finite set.
If $(c_0,z_0)\in C_{n,p}(\text{singular})$, then $Q_{0,p}(c_0,0)=0$ by point (5) of Proposition \ref{define} and point 2 of Lemma \ref{period}, whereas the roots of $Q_{0,p}(c,0)$ form a finite set.
\end{proof}

\section{The irreducible  factorization of $Q_{n,p}$ }

In this section, we will show that the curve $\mathcal{X}_{n,p}$, $n\geq1$, has $d-1$ smooth irreducible components and analyze the properties of its singular points.  We always assume $n\geq1$ without emphasizing.

\subsection{Factorization of $Q_{n,p}$ and the features of its singular points }\label{3.2}

 Recall that for $f\in\C[c,z]$, $\text{Deg}(f)$ denotes the total degree of $f$ and $\text{deg}_z(f)$ denotes the degree of the variable $z$ in $f$.
\begin{lemma}\label{factor}
\begin{description}
\item[(Algebraic version)]There exists a unique sequence of monic polynomials $\{q_{n,p}^j\in\mathbf{C}[z]\}_{1\leq j\leq d-1}$  such that $$Q_{n,p}(c,z)=\prod_{j=1}^{d-1}q_{n,p}^j(c,z).$$  All points in $C_{n,p}(\text{singular})$ are zeros of $q_{n,p}^j\in\C[c,z]$, and there are no other common zeros for  $q_{n,p}^i$ and $q_{n,p}^j$ with $i\ne j$. Moreover, we have   $\text{Deg}(q_{n,p}^j)=\text{deg}_z(q_{n,p}^j)=\nu_d(p)d^{n-1}$.
\item[(Topological version)] Let $\mathcal{V}_{n,p}^j=\{(c,z)\in\C^2\big|q_{n,p}^j(c,z)=0\}\ (1\leq j\leq d-1)$. Then $C_{n,p}(\text{singular})\subset \mathcal{V}_{n,p}^j$ for each $j$ and $\bigl\{ \mathcal{V}_{n,p}^j\setminus
C_{n,p}(\text{singular})\bigr\}_{1\leq j\leq d-1}$ are pairwise disjoint.
\end{description}
\end{lemma}
\begin{proof}
Recall that $\mathbf{C}=\C[c]$ and $\mathbf{K}$ is  a fixed {algebraically} closed field containing $\mathbf{C}$.

Let  $\Delta$ be a root of $Q_{0,p}\in\mathbf{C}[z]$.  Then by part 1 of Lemma \ref{period}, $$\Phi_{0,p}(c,\Delta)=f^p_c(\Delta)-\Delta=0.$$
We see from this equation that $\Delta$ is periodic under $f_c$ and $\Delta,\ldots, f^{p-1}_c(\Delta)$ are roots of $\Phi_{0,p}$.
Note that $\Phi_{0,p}(c,0)=f^p_c(0)$ is a polynomial in the variable $c$ of the degree $d^{p-1}$, so $\Delta\not=0$. Consequently, $\omega\Delta,\ldots,\omega^{d-1}\Delta$ are not roots of $Q_{0,p}$, where $\omega=e^{\frac{2\pi i}{d}}$, because they are not periodic under $f_c$. Then by the equation $Q_{0,p}(c,f_c(z))=Q_{0,p}(c,z)Q_{1,p}(c,z)$ (see (\ref{induction})), we get that $\omega\Delta,\ldots,\omega^{d-1}\Delta$ are roots of $Q_{1,p}\in\mathbf{C}[z]$.

 Let us factorize $Q_{0,p}$  in $\mathbf{K}$ by $$Q_{0,p}(c,z)=\prod_{i=1}^{\nu_d(p)}(z-\Delta_i)$$
($\Delta_{s_1}\ne\Delta_{s_2}$  for $s_1\ne s_2$, because all roots of $\Phi_{0,p}\in\mathbf{C}[z]$ are simple (Lemma \ref{squarefree}), and so are $Q_{0,p}$ (part 1 of Lemma \ref{period})). Then $Q_{1,p}$ can be expressed as
\begin{equation}\label{q0}
Q_{1,p}=\prod_{i=1}^{\nu_d(p)}(z-\omega\Delta_i)\cdots(z-\omega^{d-1}\Delta_i)=\prod_{j=1}^{d-1}\prod_{i=1}^{\nu_d(p)}(z-\omega^j\Delta_i)
\end{equation}
To see this, we first note that for any $s,t\in[1,d-1]$ and $i_1\not=i_2\in[1,\nu_d(p)]$, $\omega^s\Delta_{i_1}\not=\omega^t\Delta_{i_2}$. 
But it is impossible because both $\Delta_{i_1}$ and $\Delta_{i_2}$ are periodic. Thus $\{\omega\Delta_i,\ldots,\omega^{d-1}\Delta_i\}_{i=1}^{\nu_d(p)}$ are pairwise distinct roots of $Q_{1,p}\in\mathbf{C}[z]$ by the discussion above, then $\prod_{i=1}^{\nu_d(p)}(z-\omega\Delta_i)\cdots(z-\omega^{d-1}\Delta_i)$ is a divisor of $Q_{1,p}$. As its degree is $(d-1)\nu_d(p)$, equal to the degree of $Q_{1,p}$, and $Q_{1,p}$ is monic, we get (\ref{q0}).
  For $j\in[1,d-1]$, set
 \begin{equation}\label{qj}
 q_{1,p}^j(c,z)=\prod_{i=1}^{\nu_d(p)}(z-\omega^{j}\Delta_i)=(\omega^j)^{\nu_d(p)}\prod_{i=1}^{\nu_d(p)}(\omega^{-j} z-\Delta_i)=(\omega^j)^{\nu_d(p)}Q_{0,p}(c,\omega^{-j}z).
\end{equation}
 Note that $d\mid \nu_d(p)$, so $ (\omega^j)^{\nu_d(p)}=1$. Then $q^j_{1,p}(c,z)$ is a monic polynomial in $\mathbf{C}[z]$, satisfying
\begin{equation}\label{Q1}
Q_{1,p}(c,z)=\prod_{j=1}^{d-1}q_{1,p}^j(c,z).
\end{equation}
 This gives a factorization of $Q_{1,p}$ in $\mathbf{C}[z]$.  By formula (\ref{qj}) and the degree conclusion in point 1 of Lemma $\ref{period}$, the total degree $\text{Deg}(q_{1,p}^j)$ and $\text{deg}_z(q_{1,p}^j)$ are both $\nu_d(p)$.

For $n\geq2$, we can define $q_{n,p}^j(c,z)$ inductively by $q_{n,p}^j(c,z)=q_{n-1,p}^j(c,f_c(z))$. Using induction, the degree conclusion in the lemma follows directly form (2) and (3) of Lemma \ref{degree}. As $Q_{n,p}(c,z)=Q_{n-1,p}(c,f_c(z))$, we have \begin{equation}\label{e2}
Q_{n,p}(c,z)=\prod_{j=1}^{d-1}q_{n,p}^j(c,z).
\end{equation} This is a factorization of $Q_{n,p}(c,z)$ in $\mathbf{C}[z]$.

We are left to prove that each $q_{n,p}^j(c,z)$ satisfies the remaining properties announced in the lemma. For $n=1$, since $q_{1,p}^j(c,z)=Q_{0,p}(c,\omega^{-j}z)$, then $(c_0,z_0)$ is a common root of $q_{1,p}^i(c,z)$ and $q_{1,p}^j(c,z)$ for some $1\leq i\ne j\leq d-1\iff \text{both }(c_0,\omega^{-i}z_0)$ and $(c_0,\omega^{-j}z_0)$ are zeros of $Q_{0,p}(c,z)$. It follows that $\omega^{-i}z_0$ and $\omega^{-j}z_0$ are both periodic point of $f_{c_0}$, hence $z_0=0$.  Note that, in case (3) of Lemma \ref{period} item 2, the critical point $0$ is never periodic (Lemma \ref{specificity}), so $0$ has period $p$. It follows that $(c_0,z_0)\in C_{1,p}(\text{singular})$. On the other hand, if $(c_0,z_0)\in C_{1,p}(\text{singular})$, then $(c_0,\omega^{-i}z_0)=(c_0,\omega^{-j}z_0)=(c_0,0)$ is a  zero of $Q_{0,p}(c,z)$. For $n\geq2$, the conclusion can be deduced from the case of $n=1$ and the definition of $q_{n,p}^j(c,z)$.
\end{proof}

For convenience, we summarize the definitions of  $q_{1,p}^j$ in term  of $Q_{0,p}$ and  the inductive definitions of $q^j_{n,p}$ ($n\geq2$) in terms of $q_{n-1,p}^j$ as a corollary.
\begin{corollary}\label{qij}
For any $p\geq1,1\leq j\leq d-1$, and $\omega=e^{\frac{2\pi i}{d}}$, we have
\[\left\{
    \begin{array}{ll}
      q^j_{1,p}(c,z)=Q_{0,p}(c,\omega^{-j}z),  \\
      q^j_{n,p}(c,z)=q_{n-1,p}^j(c,f_c(z)), & \hbox{$n\geq2$.}
    \end{array}
  \right.
\]
\end{corollary}

\noindent{\bf Example:} Here are some examples of $Q_{n,p}$ and their decomposition. Let $d=3$. Suppose $p=1$, then we have $Q_{0,1}(c,z)=z^3+c-z$,
\begin{eqnarray*}
Q_{1,1}(c,z)&=& c^2+cz+z^2+2cz^3+z^4+z^6\\
&=&(z^3+c-e^{-\frac{2}{3}\pi i}z)(z^3+c-e^{-\frac{4}{3}\pi i}z)\\
&=&q_{1,1}^1(c,z)\cdot q_{1,1}^2(c,z).
\end{eqnarray*}
and
\begin{eqnarray*}
Q_{2,1}(c,z)&=&3c^2+3 c^4 + (c^6 + 3 c  + 10 c^3  + 6 c^5) z^3 + (1 +
 12 c^2  + 15 c^4) z^6 \\
&&+ (6 c  + 20 c^3) z^9 + (1 +
 15 c^2) z^{12} + 6 c z^{15} + z^{18}\\
&=&\big((1-e^{-\frac{2}{3}\pi i})c + c^3 + (3 c^2-e^{-\frac{2}{3}\pi i}) z^3  + 3 c z^6 + z^9\big)\times\\
&&\big((1-e^{-\frac{4}{3}\pi i})c + c^3 + (3 c^2-e^{-\frac{4}{3}\pi i}) z^3  + 3 c z^6 + z^9\big)\\
&=&q_{2,1}^1(c,z)\cdot q_{2,1}^2(c,z).
\end{eqnarray*}
Suppose $p=2$, then we have $Q_{0,2}(c,z)=1 + c^2 + c z + z^2 + 2 c z^3 + z^4 + z^6$ and
\begin{eqnarray*}
Q_{1,2}(c,z)&=&1 + 2 c^2 + c^4 - (c+c^3) - z^2 + (3 c  + 4 c^3) z^3 -
 3 c^2 z^4\\
&& + (1 + 6 c^2) z^6 - 3 c z^7 + 4 c z^9 - z^{10} + z^{12}\\
&=&\big(1+c^2+e^{-\frac{2}{3}\pi i }z+e^{-\frac{4}{3}\pi i}z^2+2cz^3+e^{-\frac{2}{3}\pi i}z^4+z^6\big)\times\\
&&\big(1+c^2+e^{-\frac{4}{3}\pi i }z+e^{-\frac{2}{3}\pi i}z^2+2cz^3+e^{-\frac{4}{3}\pi i}z^4+z^6\big)\\
&=&q_{1,2}^1(c,z)\cdot q_{1,2}^2(c,z)
\end{eqnarray*}

From Lemma \ref{factor}, we  see that in the case $d\geq3$, the polynomial $Q_{n,p}$ is both reducible and non-smooth,  because $C_{n,p}$(singular), which is non-empty, belongs to the set of singular points of $Q_{n,p}$.

We now turn to the study of the components $q_{n,p}^j(c,z)$.
The following theorem is the core of this section.

\begin{theorem}\label{irreducible}
Given $d\geq2$, for any $n,p\geq1,\ j\in[1,d-1]$,  the polynomial $q_{n,p}^j(c,z)$ is smooth and irreducible.
\end{theorem}
The proof of this theorem is postponed to $\S$ \ref{3.3}.

By this theorem, all components $\VVV_{n,p}^j$ are Riemann surfaces. Together with Lemma \ref{factor}, this implies that the singularity set of $\mathcal{X}_{n,p}$ is equal to $C_{n,p}(\text{singular})$.
 The next proposition characterizes the features of these singularities.
\begin{proposition}\label{singularity}
Given $d\geq2$, for $n,p\geq1$, each singularity $(c_0,z_0)$ of $\mathcal{X}_{n,p}$ has multiplicity $d-1$. Furthermore, if $f^l_{c_0}(z_0)=0$ for some $0\leq l\leq n-2$, then
$\mathcal{X}_{n,p}$ has one tangent of multiplicity $d-1$ at $(c_0,z_0)$; otherwise, the singularity $(c_0,z_0)$ is ordinary.
\end{proposition}
\begin{proof}
Let $(c_0,z_0)$ be a singular point of $\mathcal{X}_{n,p}$. Since each component of $\mathcal{X}_{n,p}$ is smooth and they are pairwise intersecting at $(c_0,z_0)$, then the first non-vanishing term of $Q_{n,p}(c,z)$ at $(c_0,z_0)$ is $d-1$. Hence the multiplicity of the singularity $(c_0,z_0)$ is $d-1$.

If  $n=1$, by (5) of Proposition \ref{define}, the fact that $(c_0,z_0)\in C_{1,p}(\text{singular})$ implies that $z_0=0$ and $(c_0,0)\in \mathcal{X}_{0,p}$. According to Lemma \ref{specificity} $c_0$ is not a parabolic parameter. Then it follows from  part 4 of Lemma \ref{period} that $(c_0,0)$ is not a critical point of $\pi_{0,p}$, and hence $\big(\partial Q_{0,p}/\partial z\big)(c_0,0)\not=0$. Meanwhile, according to part 5 of Lemma \ref{period}, $\big(\partial Q_{0,p}/\partial c\big)(c_0,0)\not=0$.  Thus $Q_{0,p}(c,z)$ has a local expression
\[Q_{0,p}(c,z)=a_{0,p}(c-c_0)+b_{0,p}z+\text{ higher order terms}\]
around $(c_0,0)$ with $a_{0,p},b_{0,p}\ne0$. It follows  that
\[q_{1,p}^j(c,z)=Q_{0,p}(c,\omega^{-j}z)=a_{0,p}(c-c_0)+b_{0,p}\omega^{-j}z+\text{ higher order terms}\]
Therefore the tangents of $\VVV_{1,p}^j$ ($1\leq j\leq d-1$) at $(c_0,0)$ are pairwise distinct.

For $n\geq 2$,  we denote by $a_{n,p}^j(c-c_0)+b_{n,p}^j(z-z_0)$ the equation of the tangent of $\VVV_{n,p}^j$ at $(c_0,z_0)$. By the formula $q_{n,p}^j(c,z)=q_{1,p}(c,f^{n-1}_c(z))$ (Corollary \ref{qij}), we have that
\[\left\{
    \begin{array}{ccl}
      a_{n,p}^j &=& \dfrac{\partial q_{n,p}^{j}}{\partial c}(c_0,z_0)=\dfrac{\partial q_{1,p}^{j}}{\partial c}(c_0,0)+\dfrac{\partial q_{1,p}^{j}}{\partial z}(c_0,0)\dfrac{\partial f_c^{n-1}}{\partial c}(c_0,z_0)=a_{0,p}+b_{0,p}\omega^{-j}\dfrac{\partial f_c^{n-1}}{\partial c}(c_0,z_0) \\[5pt]
      b_{n,p}^j&=& \dfrac{\partial q_{n,p}^{j}}{\partial z}(c_0,z_0)=\dfrac{\partial q_{1,p}^{j}}{\partial z}(c_0,0)(f^{n-1}_{c_0})'(z_0)=b_{0,p}\omega^{-j}(f^{n-1}_{c_0})'(z_0)
    \end{array}
  \right.
\]
If there exists $0\leq l\leq n-2$ such that $f^l_{c_0}(z_0)=0$, then $(f^{n-1}_{c_0})'(z_0)=0$, and hence $b_{n,p}^j=0$. It follows that the first non-vanishing term of $Q_{n,p}$ at $(c_0,z_0)$ is $a(c-c_0)^{d-1}$ where $a$ is a non-zero constant, i.e., $\mathcal{X}_{n,p}$ has the tangent $c=c_0$ of multiplicity $d-1$ at $(c_0,z_0)$. In  the other cases, we get $(f^{n-1}_{c_0})'(z_0)\not=0$. Combining this point and the fact  that $a_{0,p},b_{0,p}\not=0$, it is not difficult to check that the pairs $(a_{n,p}^j,b_{n,p}^j) (1\leq j\leq d-1$) are pairwise non-colinear. Hence the tangents of $\VVV_{n,p}^j$ ($1\leq j\leq d-1$) at $(c_0,z_0)$ are pairwise distinct, that is, $(c_0,z_0)$ is ordinary.
\end{proof}

\subsection{Proof of the smoothness and irreducibility of $q_{n,p}^j$}\label{3.3}
The objective here is to prove Theorem \ref{irreducible}.

 The approach to prove the smoothness is similar to that in \cite{BT}.   The idea is to prove that some partial derivative of $q_{n,p}^j$ is non vanishing. Following A. Epstein, we will express this derivative as the coefficient of a quadratic differential of the form $(f_c)_\star\Q-\Q$. Thurston's contraction principle gives $(f_c)_\star\Q-\Q\ne0$,  whence our partial derivative is non-zero.

 The approach to the irreducibility is based on the connectedness of periodic curve $\mathcal{X}_{0,p}$. Then we will show the connectivity of $\mathcal{V}_{n,p}^j$ using a branched covering
 by induction on the preperiodic index $n$.

 Here we list some definitions and results about quadratic differentials and Thurston's contraction principle. All their proofs can be found in \cite{BT} and \cite{LG}.

 We use $ \Q (\C) $ to denote the set of
meromorphic quadratic differentials on $ \C $ whose poles (if any) are all simple. If $\Q\in \Q(\C)$ and $ U $ is a bounded open subset of $ \C $, the norm
\[\| \Q \| _U \eqdef \iint_U| q| \]
is well defined and finite.

For $ f: \C \to \C $  a non-constant polynomial and $\Q=q\dz^2$ a meromorphic quadratic differential on $\C$,  the pushforward $ f_* \Q$ is defined by the quadratic differential
\[ f_* \Q: = Tq \dz ^ 2\quad\text{with}\quad  Tq (z) \eqdef \sum_{f (w) = z} \frac{q (w)}{f '(w) ^ 2}. \]
If $ Q \in \Q (\C) $, then $ f_* \Q \in \Q (\C) $ also.
The following lemma is a weak version of Thurston's contraction principle.
\begin{lemma} \label{contract}
If $ f: \C \to \C $ is a polynomial and if $ \Q \in \Q (\C) $, then
$ f_* \Q \neq \Q $.
\end{lemma}
The formulas below appeared in \cite{LG} chapter $2$, we write them together as a lemma.
\begin{lemma} [Levin] \label{levin}
For $f=f_c$, we have
\begin{equation}\label{pushforward} \left\{\begin{array}{ll} \mystrut f_*\left(\dfrac{ \dz^2}{z}\right)=0\vspace{0.1cm}\\
f_*\left(\dfrac{\dz^2}{z-a}\right)=\dfrac{1}{f'(a)}
\left(\dfrac{\dz^2}{z-f(a)}-\dfrac{\dz^2}{z-c}\right) & \text{if } a\ne 0\vspace{0.1cm}\\
\end{array}\right.\end{equation}
\end{lemma}
To prove the irreducibility of $q_{n,p}^j$, we need the following Lemma.
\begin{lemma}\label{compute}
For each $n,p\geq1$, $1\leq j\leq d-1$, the polynomial $q_{n,p}^j(c,0)$ (in the variable $c$) has degree $\nu_d(p)d^{n-2}$.
\end{lemma}
\begin{proof}
For $n=1$, we see that $q_{1,p}^j(c,0)=Q_{0,p}(c,0)$ from Corollary \ref{qij}. Then the result
follows directly from point 5 of Lemma \ref{period}.

For $n\geq2$, $q_{n,p}^j(c,0)=q^j_{1,p}(c,f^{n-1}(0))$. Since $\text{Deg}(q_{1,p}^j)=\text{deg}_z(q_{1,p}^j)=\nu_d(p)$ (see Lemma \ref{factor}) and $\text{Deg}(f^{n-1}_c(0))=d^{n-2}$ (which is easily checked), we have $$\text{Deg}(q_{n,p}^j(c,0))=\text{Deg}(q_{1,p}^j(c,z))\cdot\text{Deg}(f^{n-1}_c(0))=\nu_d(p)d^{n-2}$$ by (2) and (3) of Lemma \ref{degree}. Then the proof is completed.
\end{proof}

{\noindent\bf  Proof of Theorem \ref{irreducible}}.
The proof goes by induction on $n$.

For $n=1$, as $q_{1,p}^j(c,z)=Q_{0,p}(c,\omega^{-j}z)$ and $Q_{0,p}(c,z)$ is smooth and irreducible, we know that $q_{1,p}^j(c,z)$ are smooth and irreducible. Assume that for $1\leq l< n$, the polynomial $q_{l,p}^j(c,z)$ are smooth and irreducible. Then we will show that $q_{n,p}^j(c,z)$ are smooth and irreducible. Now fix any $j_0\in[1,d-1]$.

{\noindent\bf  Smoothness of $q_{n,p}^{j_0}$:} As $q_{n,p}^{j_0}(c,z)=q_{n-1,p}^{j_0}(c,f_c(z))$, for any $(c_0,z_0)$  a zero of $q_{n,p}^{j_0}(c,z)$, we have
\begin{equation}\label{pushforward}
\left\{\begin{array}{ll} \mystrut \dfrac{\partial q_{n,p}^{j_0}}{\partial c}(c_0,z_0)=\dfrac{\partial q_{n-1,p}^{j_0}}{\partial c}(c_0,w_0)+\dfrac{\partial q_{n-1,p}^{j_0}}{\partial z}(c_0,w_0)\vspace{0.1cm}\\
\dfrac{\partial q_{n,p}^{j_0}}{\partial z}(c_0,z_0)=\dfrac{\partial q_{n-1,p}^{j_0}}{\partial z}(c_0,w_0)\cdot f_{c_0}^\prime(z_0)\vspace{0.1cm}\\
\end{array}\right.\end{equation}
where $w_0=f_{c_0}(z_0)$. Then if $z_0\ne0$, by the smoothness of $\VVV_{n,p}^{j_0}$ (assumption of induction), $[\partial q_{n,p}^{j_0}/\partial c](c_0,z_0)$ and $[\partial q_{n,p}^{j_0}/\partial z](c_0,z_0)$ can not be equal to $0$ simultaneously, it follows that $q_{n,p}^{j_0}(c,z)$ is smooth at $(c_0,z_0)$. So we are left to prove that $q_{n,p}^{j_0}(c,z)$ is smooth at $(c_0,0)\in\VVV_{n,p}^{j_0}$. In this situation, $c_0$ is either a $p$-periodic super-attracting parameter or a $(n,p)$-Misiurewicz parameter, and $[\partial q_{n,p}^{j_0}/\partial z](c_0,0)=0$. So we have to show $[\partial q_{n,p}^{j_0}/\partial c](c_0,0)\ne0$.

In the former case $f_{c_0}^{n-1}(0)=0$, then $p|n-1$. Since $q_{n,p}^{j_0}(c,z)=q_{1,p}^{j_0}(c,f_c^{n-1}(z))$, we have
\begin{equation}\label{eq1}
\frac{\partial q_{n,p}^{j_0}}{\partial c}(c_0,0)=\frac{\partial q_{1,p}^{j_0}}{\partial c}(c_0,0)+\frac{\partial q_{1,p}^{j_0}}{\partial z}(c_0,0)\frac{\partial f_c^{n-1}}{\partial c}(c_0,0)
\end{equation}
Note that $Q_{0,p}(c_0,0)=0$ and $p|n-1$, then differentiating both sides of the equation
\[f_c^{n-1}(z)-z=\prod_{k|n-1}Q_{0,k}(c,z),\]
which is raised in point 1 of Lemma \ref{period}, with respect to $c$ and $z$ respectively at the point $(c_0,0)$, we have that
\begin{equation}\label{eq2}
\left\{
  \begin{array}{ccc}
    \dfrac{\partial f_c^{n-1}}{\partial c}(c_0,0)&=&\dfrac{\partial Q_{0,p}}{\partial c}(c_0,0)\ds\prod_{{k|n-1}\atop{k\not=p}}Q_{0,k}(c_0,0) \\[5pt]
    -1&=&\dfrac{\partial Q_{0,p}}{\partial z}(c_0,0)\ds\prod_{{k|n-1}\atop{k\not=p}}Q_{0,k}(c_0,0)
  \end{array}
\right.
\end{equation}
Since $q_{1,p}^{j_0}(c,z)=Q_{0,p}(c,\omega^{-j_0}z)$, then
\[\dfrac{\partial q_{1,p}^{j_0}}{\partial c}(c_0,0)=\dfrac{\partial Q_{0,p}}{\partial c}(c_0,0),\quad \dfrac{\partial q_{1,p}^{j_0}}{\partial z}(c_0,0)=\omega^{-j_0}\dfrac{\partial Q_{0,p}}{\partial z}(c_0,0).\]
By substituting these two formulas into equation (\ref{eq1}) and applying equation (\ref{eq2}), we find
\begin{eqnarray}\label{eq3}
\frac{\partial q_{n,p}^{j_0}}{\partial c}(c_0,0)&=&\frac{\partial Q_{0,p}}{\partial c}(c_0,0)+\omega^{-j_0}\dfrac{\partial Q_{0,p}}{\partial z}(c_0,0)\dfrac{\partial Q_{0,p}}{\partial c}(c_0,0)\ds\prod_{{k|n-1}\atop{k\not=p}}Q_{0,k}(c_0,0)\nonumber\\
&=&\frac{\partial Q_{0,p}}{\partial c}(c_0,0)\big(1+\omega^{-j_0}\dfrac{\partial Q_{0,p}}{\partial z}(c_0,0)\ds\prod_{{k|n-1}\atop{k\not=p}}Q_{0,k}(c_0,0)\big)\nonumber\\
&=&\frac{\partial Q_{0,p}}{\partial c}(c_0,0)(1-\omega^{-j_0}).
\end{eqnarray}
By point 5 of Lemma \ref{period}, $[\partial Q_{0,p}/\partial c](c_0,0)\not=0$, then so is $[\partial q_{n,p}^{j_0}/\partial c](c_0,0)$.

In the latter case, since
\[\dfrac{\partial Q_{n,p}}{\partial c}(c_0,0)=\prod_{1\leq j\ne j_0\leq d-1}q_{n,p}^j(c_0,0)\cdot\dfrac{\partial q_{n,p}^{j_0}}{\partial c}(c_0,0)\]
and the point $(c_0,0)$ is not a zero of $\prod_{j\ne j_0}q_{n,p}^j(c,z)$ by Lemma \ref{factor}, we only have to show $[\partial Q_{n,p}/\partial c](c_0,0)\ne0$. Furthermore, since
\[\dfrac{\partial \Phi_{n,p}}{\partial c}(c_0,0)=\Phi_{n-1,p}(c_0,0)\cdot\prod_{k|p,k<p}Q_{n,k}(c_0,0)\cdot\dfrac{\partial Q_{n,p}}{\partial c}(c_0,0)\]
and $\Phi_{n-1,p}(c_0,0)\cdot\prod_{k|p,k<p}Q_{n,k}(c_0,0)\ne0$, it is equivalent to show $[\partial \Phi_{n,p}/\partial c](c_0,0)\ne0$. We shall choose a meromorphic quadratic differential with simple poles such that
\[(f_{c_0})_*\Q=\Q+\dfrac{\partial \Phi_{n,p}}{\partial c}(c_0,0)\cdot\dfrac{dz^2}{z-c_0}.\]
Then by Lemma \ref{contract}, we obtain  $[\partial \Phi_{n,p}/\partial c](c_0,0)\ne0$.

We shall use the following notations:
\begin{eqnarray*}
z_k:=f_{c_0}^{\circ n+k}(0),\ &\delta_k:=f_{c_0}^\prime(z_k)=dz_k^{d-1},&\quad 0\leq k\leq p-1\\
y_{l}:=f_{c_0}^l(0),\quad\ &\varepsilon_l:=f_{c_0}^\prime(y_{l})=dy_{l}^{d-1},&\quad 1\leq l\leq n-1
\end{eqnarray*}
With these notations and a bit of calculations, we get
\begin{eqnarray*}
\dfrac{\partial \Phi_{n,p}}{\partial c}(c_0,0)&=&\dfrac{\partial f_c^{\circ(n+p)}}{\partial c}(c_0,0)-\dfrac{\partial f_c^{\circ n}}{\partial c}(c_0,0)\\
                                                                          &=&(\delta_0\cdots\delta_{p-1}-1)(\varepsilon_{n-1}\cdots\varepsilon_1+\cdots+\varepsilon_{n-1}\varepsilon_{n-2}+\varepsilon_{n-1}+1)\\[6pt]
                                                                          &&{}+\delta_{p-1}\cdots\delta_1+\cdots+\delta_{p-1}+1
\end{eqnarray*}
Denote $(\delta_0\cdots\delta_{p-1}-1)(\varepsilon_{n-1}\cdots\varepsilon_1+\cdots+\varepsilon_{n-1}\varepsilon_{n-2}+\varepsilon_{n-1}+1)$ by $\alpha$.
Let
\[\Q=\sum_{k=0}^{p-1}\dfrac{\rho_k}{z-z_k}dz^2+\sum_{l=1}^{n-1}\dfrac{\lambda_l}{z-y_l}dz^2\]
be a quadratic differential in $\Q(\C)$. Here $\rho_k\ (0\leq k\leq p-1),\ \lambda_l\ (1\leq l\leq n-1)$ are undetermined coefficients (note that $y_1=c_0$). Applying Lemma \ref{levin} and writing $f$ for $f_{c_0}$, we have
\begin{eqnarray*}
f_*\Q&=&\sum_{k=0}^{p-1}\dfrac{\rho_k}{\delta_k}\left(\dfrac{dz^2}{z-z_{_{k+1}}}-\dfrac{dz^2}{z-c_0}\right)+\sum_{l=1}^{n-2}\dfrac{\lambda_l}{\varepsilon_l}\left(\dfrac{dz^2}{z-y_{_{l+1}}}-\dfrac{dz^2}{z-c_0}\right)+\dfrac{\lambda_{n-1}}{\varepsilon_{n-1}}\left(\dfrac{dz^2}{z-z_0}-\dfrac{dz^2}{z-c_0}\right)\\
         &=&\left(\dfrac{\rho_{p-1}}{\delta_{p-1}}+\dfrac{\lambda_{n-1}}{\varepsilon_{n-1}}\right)\dfrac{dz^2}{z-z_0}+\dfrac{\rho_0}{\delta_0}\dfrac{dz^2}{z-z_1}+\cdots+\dfrac{\rho_{p-2}}{\delta_{p-2}}\dfrac{dz^2}{z-z_{p-1}}\\
         & &{}+\left(\alpha-\sum_{l=1}^{n-1}\dfrac{\lambda_l}{\varepsilon_l}\right)\dfrac{dz^2}{z-y_1}+\dfrac{\lambda_1}{\varepsilon_1}\dfrac{dz^2}{z-y_2}+\cdots+\dfrac{\lambda_{n-2}}{\varepsilon_{n-2}}\dfrac{dz^2}{z-y_{n-1}}-\left(\alpha+\sum_{k=0}^{p-1}\dfrac{\rho_k}{\delta_k}\right)\dfrac{dz^2}{z-c_0}
          \end{eqnarray*}
We want to choose $\Q$ so that $$f_*\Q-\Q=-\left(\alpha+\sum_{k=0}^{p-1}\dfrac{\rho_k}{\delta_k}\right)\dfrac{dz^2}{z-c_0}$$
It amounts then to solve the following linear system on the unknown coefficient vector $(\rho_0,\ldots,\rho_{p-1},\lambda_1,\ldots,\lambda_{n-1}):$
\[\left(\begin{array}{cccccccccccc}
\frac{1}{\delta_0}&-1&\\
                                &\ \ \cdot&\ \ \cdot\\
                                 & &\ \ \cdot&\ \ \cdot\\
                                 & & &\ \ \cdot&\ \ \cdot\\
                                 & & & &\frac{1}{\delta_{_{p-2}}}&-1\\
                              -1& & & & &\frac{1}{\delta_{_{p-1}}}& & & & & &\frac{1}{\varepsilon_{_{n-1}}}\\
                                 & & & & & &1+\frac{1}{\varepsilon_1}&\frac{1}{\varepsilon_2}&\frac{1}{\varepsilon_3}&\cdots&\frac{1}{\varepsilon_{_{n-2}}}&\frac{1}{\varepsilon_{_{n-1}}}\\
                                 & & & & & &\frac{1}{\varepsilon_1}&-1\\
                                 & & & & & & &\ \ \cdot&\ \ \cdot\\
                                 & & & & & & & &\ \ \cdot&\ \ \cdot\\
                                 & & & & & & & & &\ \ \cdot&\ \ \cdot\\
                                 & & & & & & & & & &\frac{1}{\varepsilon_{_{n-2}}}&-1
                                 \end{array}\right)
   \left(\begin{array}{c}
   \rho_0\\
   \cdot\\
   \cdot\\
   \cdot\\
   \rho_{_{p-2}}\\
   \rho_{_{p-1}}\\
   \lambda_1\\
   \lambda_2\\
   \cdot\\
   \cdot\\
   \cdot\\
   \lambda_{_{n-1}}
   \end{array} \right)=
   \left(\begin{array}{c}
   0\\
   \cdot\\
   \cdot\\
   \cdot\\
   0\\
   0\\
   \alpha\\
   0\\
   \cdot\\
   \cdot\\
   \cdot\\
   0\\
   \end{array}\right)
\]
Denote by $A$ the coefficient matrix, we have
\[\text{det}(A)=\dfrac{(-1)^{n-1}\alpha}{\delta_0\cdots\delta_{p-1}\cdot\varepsilon_1\cdots\varepsilon_{n-1}}\]
Then whether $\alpha=0$ or not, this linear system has   non-zero solutions, and one of its solutions is
\begin{equation}\label{solution}
\left\{\begin{array}{lll}
\rho_0&=&\delta_0\cdots\delta_{p-1}\\
\rho_1&=&\delta_1\cdots\delta_{p-1}\\
\vdots\\
\rho_{p-1}&=&\delta_{p-1}\\
\lambda_1&=&(\delta_0\cdots\delta_{p-1}-1)\cdot\varepsilon_{n-1}\cdots\varepsilon_1\\
\vdots\\
\lambda_{n-2}&=&(\delta_0\cdots\delta_{p-1}-1)\cdot\varepsilon_{n-1}\varepsilon_{n-2}\\
\lambda_{n-1}&=&(\delta_0\cdots\delta_{p-1}-1)\cdot\varepsilon_{n-1}
\end{array}\right.
\end{equation}
Therefore, for $(\rho_0,\ldots,\rho_{p-1},\lambda_1,\ldots,\lambda_{n-1})$ satisfies (\ref{solution}), we have
\begin{eqnarray*}
f_*\Q-\Q&=&-\left(\alpha+\sum_{k=0}^{p-1}\dfrac{\rho_k}{\delta_k}\right)\dfrac{dz^2}{z-c_0}=-\dfrac{\partial \Phi_{n,p}}{\partial c}(c_0,0)\cdot \dfrac{dz^2}{z-c_0}
\end{eqnarray*}
As a consequence $[\partial \Phi_{n,p}/\partial c](c_0,0)\ne 0$.

{\noindent\bf Irreducibility of $q_{n,p}^{j_0}$:}  For $n\geq2$, $q_{n,p}^j(c,z)$ is defined by $q_{n,p}^j(c,z)=q_{n-1,p}^j(c,f_c(z))$. Interpreting these equations from a topological view, we obtain a sequence of maps
  $$\Big\{\wp_{n,p}^j:\mathcal{V}_{n,p}^j\longrightarrow \VVV_{n-1,p}^j,\ (c,z)\mapsto (c,f_c(z)) \mid n\geq2,\ p\geq1,\ 1\leq j\leq d-1\Big\}.$$ Note that for $n=1$, we can also define a map $\wp_{1,p}^j: \VVV_{1,p}^j\rightarrow \mathcal{X}_{0,p}$ by $\wp_{1,p}^j(c,z)=(c,f_c(z))$. By the smoothness of $\VVV_{n,p}^j$, we can check the following results.

\begin{itemize}
\item The map $\wp_{1,p}^j:\VVV_{1,p}^j\rightarrow \mathcal{X}_{0,p}$ is a homeomorphism. To see this, notice that $q_{1,p}^j(c,z)=Q_{0,p}(c,\omega^{-j}z)$ (Corollary \ref{qij}), so we can define a map $\phi_{1,p}^j$ from $\mathcal{X}_{0,p}$ to $\VVV_{1,p}^j$ by mapping a point $(c_0,w_0)\in \mathcal{X}_{0,p}$ to $(c_0,\omega^jz_0)\in \VVV_{1,p}^j$, where $z_0$ is the point in the orbit of $w_0$ under $f_{c_0}$ with $f_{c_0}(z_0)=w_0$. By a simple computation, we can see that $\phi_{1,p}^j\circ\wp_{1,p}^j=id_{\VVV_{1,p}^j}$ and $\wp_{1,p}^j\circ\phi_{1,p}^j=id_{\mathcal{X}_{0,p}}$. Hence $\wp_{1,p}^j$ is a homeomorphism.
\item For $n\geq2$, the map $\wp_{n,p}^j:\mathcal{V}_{n,p}^j\rightarrow \VVV_{n-1,p}^j$ is a degree $d$ branched covering with critical  set $$D_{n,p}^j=\bigl\{(c,0)\mid q_{n,p}^j(c,0)=0\bigr\}.$$
and each critical point has multiplicity $d-1$.

In fact,  a point $(c_0,w_0)\in \VVV_{n-1,p}^j\setminus \wp(D_{n,p}^j)$ has $d$ preimages $(c_0,z_1),\ldots,(c_0,z_d)$ under $\wp_{n,p}^j$, where $z_1,\ldots,z_d$ are preimages of $w_0$ under $f_{c_0}$. Fix $i\in[1,d]$. If $[\partial q_{n,p}^j/\partial z](c_0,z_i)\not=0$, then by equation (\ref{pushforward}), $[\partial q_{n-1,p}^j/\partial z](c_0,w_0)\not=0$. It implies that some neighborhoods of $(c_0,z_i)$ and $(c_0,w_0)$ can be parameterized by $c$ respectively. Using such two local coordinates, the map $\wp_{n,p}^j$ has a local expression $c\mapsto c$ near $(c_0,z_i)$, which means that $\wp_{n,p}^j$ is a local homeomorphism near $(c_0,z_i)$. If $[\partial q_{n,p}^j/\partial z](c_0,z_i)=0$, then by equation (\ref{pushforward}), the fact of $z_i\not=0$ and the smoothness of $q_{n,p}^j$, we have that $[\partial q_{n-1,p}^j/\partial z](c_0,w_0)=0$ and
\[\dfrac{\partial q_{n,p}^{j}}{\partial c}(c_0,z_i)=\dfrac{\partial q_{n-1,p}^{j}}{\partial c}(c_0,w_0)\not=0\]
It implies that some neighborhoods of $(c_0,z_i)$ and $(c_0,w_0)$ can be parameterized by $z$ respectively, and $c'(z_i)=0$.  Using such two local coordinates, the map $\wp_{n,p}^j$ has a local expression $z\mapsto f_{c(z)}(z)$ near $(c_0,z_i)$. Since $z_i\not=0$, then $\dfrac{df_{c(z)}(z)}{dz}\big|_{z=z_i}=dz_i\not=0$, which still means that $\wp_{n,p}^j$ is a local homeomorphism near $(c_0,z_i)$.

By the discussion above, we can see that $$\wp_{n,p}^j:\VVV_{n,p}^j\setminus (\wp_{n,p}^j)^{-1}(\wp(D_{n,p}^j))\to \VVV_{n-1,p}^j\setminus \wp(D_{n,p}^j)$$
 is a degree $d$ covering. On the other hand, for any point in $\wp_{n,p}^j(D_{n,p}^j)$, it has only one preimage, which belongs to $D_{n,p}^j$. Hence we have that $\wp:\VVV_{n,p}^j\to\VVV_{n-1,p}^j$ is a degree $d$ branched covering (because $(\wp_{n,p}^j)^{-1}(\wp(D_{n,p}^j))=D_{n,p}^j$ and $D_{n,p}^j$ is finite) and the local degree of $\wp_{n,p}^j$ at each point of $D_{n,p}^j$ is $d$.
\end{itemize}
By the smoothness of $q_{n,p}^{j_0}(c,z)$ and the inductive assumption of irreducibility, we know that
 $\VVV_{n-1,p}^{j_0}$ and each connected component of $\VVV_{n,p}^{j_0}$ is a Riemann surface. Then  the restriction of $\wp_{n,p}^{j_0}$ on any connected component of $\VVV_{n,p}^{j_0}$ is also a branched covering.
Lemma \ref{compute} implies that the critical set $D_{n,p}^{j_0}$ of $\wp_{n,p}^{j_0}$
is non-empty. Since each critical point has multiplicity $d-1$, the set  $\VVV_{n,p}^{j_0}$ must be connected. By Lemma \ref{relation} and the smoothness of $q_{n,p}^{j_0}$, we conclude that $q_{n,p}^{j_0}(c,z)$ is irreducible in $\C[c,z]$.

\hspace*{\fill}$\square$

\section{Genus of the compactification of $\mathcal{V}_{n,p}^j$}

In the previous section, we have seen that $\mathcal{X}_{n,p}$ consists of  $d-1$ Riemann surfaces $\mathcal{V}_{n,p}^j$ which are pairwise  intersecting at the singular points of $\mathcal{X}_{n,p}$. In order to give a complete topological description of $\mathcal{X}_{n,p}$, we also need the topological characterization of each $\mathcal{V}_{n,p}^j$.

In fact, by adding an ideal boundary point at each end of $\mathcal{V}_{n,p}^j$, we obtain a compactification of $\mathcal{V}_{n,p}^j$, denoted by $\widehat{\mathcal{V}}_{n,p}^j$, such that $\widehat{\mathcal{V}}_{n,p}^j$
is a compact Riemann surface (in $\S$ \ref{4.2} ).  The genus of  $\widehat{\mathcal{V}}_{n,p}^j$ is calculated (in $\S$ \ref{4.3}). Topologically, $\mathcal{X}_{n,p}$ is  completely determined by the number of its singular points,  the genus of  $\widehat{\mathcal{V}}_{n,p}^j$ and the number of ideal boundary points added on $\mathcal{V}_{n,p}^j$ (or the number of ends of $\mathcal{V}_{n,p}^j$).

\subsection{Compactification of $\mathcal{V}_{n,p}^j$}\label{4.2}

Denote by $\pi_{n,p}^j\colon \mathcal{V}_{n,p}^j\to \C$ the projection from $\mathcal{V}_{n,p}^j$ to the parameter plane, i.e., $\pi_{n,p}^j(c,z)=c$. It is easy to see
\begin{equation}\label{map}
\pi_{n,p}^j=\pi_{0,p}\circ\wp_{1,p}^j\circ\cdots\circ\wp_{n-1,p}^j\circ\wp_{n,p}^j
\end{equation}
where  $\pi_{0,p}$ is the projection from $\mathcal{X}_{0,p}$ to the parameter plane and $\wp_{n,p}^j$ is defined in the proof of irreducibility.
 It follows that  $\pi_{n,p}^j$
is a degree $\nu_d(p)d^{n-1}$ branched covering.
To study the critical points of $\pi_{n,p}^j$, we define a subset $C_{n,p}^{\text{crit}}(\text{singular})$ of $C_{n,p}(\text{singular})$ by
\begin{equation}\label{singularcritical}
C_{n,p}^{\text{crit}}(\text{singular})=\{(c,z)\in C_{n,p}(\text{singular})\mid f^l_c(z)=0\text{ for some $0\leq l\leq n-2$}\}
\end{equation}

\begin{lemma}\label{critical}
For any $l,p\geq1$, the critical set of $\pi_{l,p}^j$ is the union of $C_{l,p}^j(\text{primitive})$, $C_{l,p}^j(\text{satellite})$, $C_{l,p}^j(\text{Misiurewicz})$ and $C_{l,p}^{\text{crit}}(\text{singular})$, where $C_{l,p}^j(\text{M}):=C_{l,p}(\text{M})\cap \mathcal{V}_{l,p}^j$ and M indicates
different properties.
\end{lemma}
\begin{proof}
We first note that $(c_0,z_0)$ is a critical point of $\pi_{l,p}^j$ if and only if $[\partial q_{l,p}^j/\partial z](c_0,z_0)=0$. By part 4 of Lemma \ref{period} and the fact that $\wp_{1,p}^j$ is homeomorphic (which is shown in the proof of irreducibility of $q_{l,p}^j$), the critical set of $\pi_{1,p}^j$ is $C_{1,p}^j(\text{primitive})\cup C_{1,p}^j(\text{satellite})$. In the case $l=1$, $C_{l,p}(\text{Misiurewicz})$ and $C_{l,p}^{\text{crit}}(\text{singular})$ are empty.

 For $l\geq 2$, by Corollary \ref{qij}, we have $q_{l,p}^j(c,z)=q_{1,p}^j(c,f^{l-1}_c(z))$. Then a point $(c_0,z_0)$ is critical for $\pi_{l,p}^j$ if and only if
\[\dfrac{\partial q_{l,p}^j}{\partial z}(c_0,z_0)=\dfrac{\partial q_{1,p}^j}{\partial z}(c_0,f^{l-1}_{c_0}(z_0))\cdot(f_{c_0}^{l-1})'(z_0)=0\]
It is equivalent that either $(c_0,f_{c_0}^{l-1}(z_0))$ is a critical point of $\wp_{1,p}^j$ or $f_{c_0}^l(z_0)=0$ for some $0\leq q\leq n-2$. By Proposition \ref{define}, the former case happens if and only if $(c_0,z_0)\in C_{l,p}^j(\text{primitive})\cup C_{l,p}^j(\text{satellite})$, and the latter case happens if and only if $(c_0,z_0)\in C_{l,p}^j(\text{Misiurewicz})\cup C_{l,p}^{\text{crit}}(\text{singular})$.
\end{proof}

 From this Lemma, we see that the critical value set of $\pi_{n,p}^j$ is contained in the union of parabolic, super-attracting and Misiurewicz parameters.
Hence $\C\setminus M_d$ contains no critical values. It follows that  each connected component of $(\wp_{n,p}^j)^{-1}(\C\setminus M_d)$, called an $end$ of $\VVV_{n,p}^j$, is conformal to $\C\setminus \overline{\D}$. By adding an ideal boundary point at the infinitely far boundary, each end of $\VVV_{n,p}^j$ is conformal to the unit disk, and then $\VVV_{n,p}^j$ becomes a compact Riemann surface. This gives a kind of compactification of $\VVV_{n,p}^j$ and we will calculate in the next subsection the genus of this compact Riemann surface.

More precisely, set $\{\ \mathcal{E}_{n,p,i}^j\}$ ($1\leq i\leq m_{n,p}^j$)  the ends of $\VVV_{n,p}^j$. Denote by $\infty_{n,p,i}^j$ the  point added at the infinitely far boundary of $\EEE_{n,p,i}^j$. Then the surface $\widehat{\mathcal{V}}_{n,p}^j:=\mathcal{V}_{n,p}^j\cup \{\infty_{n,p,i}^j\}_{i=1}^{m_{n,p}^j}$ is a compactification of $\VVV_{n,p}^j$ and $\widehat{\EEE}_{n,p,i}^j:=\EEE_{n,p,i}^j\cup\{\infty_{n,p,i}^j\}$ is  called an {\it end} of $\widehat{\VVV}_{n,p}^j$. In this case, the map $\pi_{n,p}^j$ can be extended to
 \[\widehat{\pi}_{n,p}^j: \widehat{\VVV}_{n,p}^j\longrightarrow\widehat{\C}\]
 by setting $\widehat{\pi}_{n,p}^j(\infty_{n,p,i}^j)=\infty$.

 \subsection{Calculation of the genus of $\widehat{\mathcal{V}}_{n,p}^j$}\label{4.3}

Now, for any $n,p\geq1,\ j\in[1,d-1]$, we have obtained a branched covering $\widehat{\pi}_{n,p}^j\colon \widehat{\mathcal{V}}_{n,p}^j\to \widehat{\C}$ of degree $\nu_d(p)d^{n-1}$ between two compact Riemann surface. By the Riemann-Hurwitz formula, we have
\[2-2g_{n,p}^j+\text{ total number of critical points of }\widehat{\pi}_{n,p}^j=2\nu_d(p)d^{n-1}.\]
where $g_{n,p}^j$ denotes the genus of $\widehat{\mathcal{V}}_{n,p}^j$.  So in order to calculate the genus of $\widehat{\mathcal{V}}_{n,p}^j$, we only need to count the number of  critical points of $\widehat{\pi}_{n,p}^j$ counting with multiplicity.  By Lemma \ref{critical}, we know that the critical points of $\widehat{\pi}_{n,p}^j$ consists of the points of $C_{n,p}^j(\text{primitive})$, $C_{n,p}^j(\text{satellite})$, $C_{n,p}^j(\text{Misiurewicz})$, $C_{n,p}^{\text{crit}}(\text{singular})$ and maybe some added ideal boundary points. So  we will count them class by class.

\begin{itemize}
\item Counting the points of $C_{n,p}^j(\text{primitive})$ and $C_{n,p}^j(\text{satellite})$.

In \cite{B}, Bousch counts the number of critical points in $C_{0,p}(\text{primitive})$ and $C_{0,p}(\text{satellite})$. His argument can be directly extended to our case (see also \cite[Thm. 4.17]{Sil}).
so we only list the result without the counting process. The number of critical points  counted with multiplicity of $\widehat{\pi}_{n,p}^j$ in $C_{n,p}^j(\text{primitive})$ and $C_{n,p}^j(\text{satellite})$ are
\[d^{n-1}p\bigl[(d-1)\nu_d(p)/d-\sum_{k|p,k<p}(\nu_d(k)/d)(d-1)\varphi(p/k)\bigr]\]
and
\[d^{n-1}\sum_{k|p,k<p}(\nu_d(k)/d)(d-1)\varphi(p/k)k(p/k-1).\]

\item Counting the points of $C_{n,p}^j(\text{Misiurewicz})$.

Recall that $D_{s,p}^j=\{(c,0)\in\C^2\big|q_{s,p}^j(c,0)=0\},s\geq 2,$ is the set of critical points of $\wp_{s,p}^j$. By Proposition \ref{define}, if $(c,0)\in D_{s,p}^j$, then $c$ is either a $(s,p)$-Misiurewicz parameter or a $p$-super-attracting parameter. So we divide the set $D_{s,p}^j$ into two sets
$$D_{s,p}^j(\text{Misiurewicz})=\{(c,0)\in D_{s,p}^j\mid \text{c is a Misiurewicz parameter}\}$$ and
$$D_{s,p}^j(\text{period})=\{(c,0)\in D_{s,p}^j\mid \text{c is a super-attracting parameter}\}$$
By the definition of $C_{n,p}^j(\text{Misiurewicz})$, we have
$$C_{n,p}^j(\text{Misiurewicz})=\bigcup_{s=2}^n(h_{n,s,p}^j)^{-1}(D_{s,p}^j(\text{Misiurewicz})),$$ where
$h_{n,s,p}^j:=\wp_{s+1,p}^j\circ\cdots\circ\wp_{n,p}^j:\mathcal{V}_{n,p}^j\longrightarrow \mathcal{V}_{s,p}^j$

Fix any $s\in[2,n]$. Since  the degree of $q_{s,p}^j(c,0)$ is $\nu_d(p)d^{s-2}$ (Lemma \ref{compute}) and  $[\partial q_{s,p}^j/\partial c](c,0)\not=0$ at each $(c,0)\in D_{s,p}^j$ (this point is shown in the proof of  smoothness of $q_{s,p}^j(c,z)$), we get $\#D_{s,p}^j=\nu_{d}(p)d^{s-2}$. By point (5) of Proposition \ref{define},  the set $D_{s,p}^j(\text{period})$ is non-empty if and only if $p|s-1$. In this case, we also see that $D_{s,p}^j(\text{period})=\{(c,0)\mid Q_{0,p}(c,0)=0\}$, then $\# D_{s,p}^j(\text{period})$ equals to $\nu_d(p)/d$ if $p|s-1$ and $0$ otherwise. It follows that
\[\# D_{s,p}^j(\text{Misiurewicz})=\left\{
                                     \begin{array}{ll}
                                       \nu_d(p)d^{s-2}, & \hbox{if $p\nmid s-1$;} \\
                                       \nu_d(p)d^{s-2}-\nu_d(p)/d, & \hbox{if $p| s-1$.}
                                     \end{array}
                                   \right.
\]
 Note that the critical value set of $h_{n,s,p}^j$ is disjoint from $D_{s,p}^j(\text{Misiurewicz})$, so $$\#(h_{n,s,p}^j)^{-1}(D_{s,p}^j(\text{Misiurewicz}))=\#D_{s,p}^j(\text{Misiurewicz})\cdot d^{n-s}$$ and each point in $(h_{n,s,p}^j)^{-1}(D_{s,p}^j(\text{Misiurewicz}))$ is a critical point of $\widehat{\pi}_{n,p}^j$ with multiplicity $d-1$.  Therefore the  number of critical points counting with multiplicity of $\widehat{\pi}_{n,p}^j$ in $C_{n,p}^j(\text{Misiurewicz})$, denoted by $M_{n,p}$, is equal to
\begin{eqnarray}\label{number1}
M_{n,p}:&=&\sum_{s=2}^n\#D_{s,p}^j(\text{Misiurewicz})\cdot d^{n-s}\cdot(d-1)\nonumber\\
&=&\nu_d(p)d^{n-2}(d-1)\big(n-1-\sum_{t=1}^{[\frac{n-1}{p}]}d^{-tp}\big).
\end{eqnarray}

\item Counting the points of $C_{n,p}^{\text{crit}}(\text{singular})$.

   Recall that $C_{n,p}^{\text{crit}}(\text{singular})$ consists of points of the form $(c,z)$ with $f_c^{n-1}(z)=0$ and such that there exists $l$ between $0$ and $n-2$ (both included) with $f^l(z)=0$ and s.t. $0$ is $p$-periodic.

 We divide the set $C_{n,p}(\text{singular})$
into several  subsets $C_{n,p}^t(\text{singular})$ which consists of points $(c,z)\in C_{n,p}(\text{singular})$ such that
\[ n-1-tp=\min\{\ l\mid f^l_c(z)=0\ \}\]
The index $t$  can take the values $0,\ldots, [\frac{n-1}{p}]$, where $[x]$ denotes the maximal integer  less than or equal to $x$, and the sets $C_{n,p}^t$ are pairwise disjoint and form a partition of $C_{n,p}(\text{singular})$.  From (\ref{singularcritical}), we see that $C_{n,p}^{\text{crit}}(\text{singular})$ is the union of  $C_{n,p}^t(\text{singular}), t\geq1$. Then we get $\#C_{n,p}^{\text{crit}}(\text{singular})=0$ if $n-1< p$. So in the following discussion, we only concern the case of $n-1\geq p$, i,e., $[\frac{n-1}{p}]\geq 1$.

Let $t\geq1$. A point $(c,z)\in C_{n,p}^t(\text{sigular})$ if and only if $(c,0)\in D_{tp+1,p}^j(\text{period})$, $f_c^{n-1-tp}(z)=0$ and $(f_c^{n-1-tp})'(z)\not=0$. Hence $$C_{n,p}^t(\text{sigular})=(h_{n,tp+1,p}^j)^{-1}(D_{tp+1,p}^j(\text{period}))\setminus(h_{n,(t+1)p+1,p}^j)^{-1}(D_{(t+1)p+1,p}^j(\text{period}))$$
if $(t+1)p+1\leq n$, and $C_{n,p}^t(\text{sigular})=(h_{n,tp+1,p}^j)^{-1}(D_{tp+1,p}^j(\text{period}))$ otherwise. So we have
\[\#C_{n,p}^t(\text{sigular})=\left\{
                         \begin{array}{ll}
                           d^{n-1-tp}\cdot \nu_d(p)/d, & \hbox{ if $t=[\frac{n-1}{p}]$;} \\
                           d^{n-1-tp}\cdot \nu_d(p)/d-d\cdot d^{n-1-(t+1)p}\cdot \nu_d(p)/d, & \hbox{ if $1\leq t<[\frac{n-1}{p}]$.}
                         \end{array}
                       \right.
\]
On the other hand, the map $h_{n,tp+1,p}^j:\VVV_{n,p}^j\to \VVV_{tp+1,p}^j$ is injective in a neighborhood of any point $(c,z)\in C_{n,p}^t(\text{sigular})$, and the map $\pi_{kp+1,p}^j:\VVV_{tp+1}^j\to\C$ has the local degree $d^t$ at the point $(c,0)$, so the number of critical points counting with multiplicity in $C_{n,p}^t(\text{sigular})$ is $(d^{t}-1)\#C_{n,p}^t(\text{sigular})$. Then the total number of critical points counting with multiplicity in $C_{n,p}(\text{sigular})$, in the case of $[\frac{n-1}{p}]\geq 1$, is
\begin{eqnarray}\label{number2}
K_{n,p}:&=&\sum_{t=1}^{[\frac{n-1}{p}]}(d^{t}-1)\#C_{n,p}^t(\text{sigular})\nonumber\\
&=&\nu_d(p)(d^{p-1}-1)d^{n-1-p}(\xi_{n,p}-\zeta_{n,p})+(d^{[\frac{n-1}{p}]}-1)\nu_d(p)d^{n-2-[\frac{n-1}{p}]p}
\end{eqnarray}
where $\xi_{n,p}:=\sum_{t=1}^{[\frac{n-1}{p}]-1} d^{-t(p-1)}$ and $\zeta_{n,p}:=\sum_{t=1}^{[\frac{n-1}{p}]-1}d^{-pt}$.

Note that when $[\frac{n-1}{p}]=0$, the number computed by formula (\ref{number2}) is $0$, which is still equal to the number of $C_{n,p}^{\text{crit}}(\text{singular})$. So the number $K_{n,p}$, defined by (\ref{number2}), is equal to the number of critical points
counting with multiplicity in $C_{n,p}^{\text{crit}}(\text{singular})$ in all cases.

\item Counting the ideal boundary points.

In \cite{B}, \cite{Mil1}, Bousch and Milnor show that the local degree of $\pi_{0,p}$ at each ideal boundary point is  $2$ (in the case of $d=2$) by analysing
the asymptotic behavior of $f_c(z)$ as $(c,z)$ goes to an ideal boundary point on $\mathcal{X}_{0,p}$. Their argument can be easily generalized to our case with degree $d\geq2$. Just to be self-contained we give an alternative  proof using the monodromy action (Lemma \ref{compactification} below).
 By Lemma \ref{compactification}, the local degree of $\widehat{\mathcal{V}}_{n,p}^j$ at each ideal boundary point is $d$. It follows that the number of ideal boundary points is $\nu_d(p)d^{n-2}$ because $\widehat{\pi}_{n,p}^j$ is a degree $\nu_d(p)d^{p-1}$ branched covering.
So the number of critical points counting with multiplicity is equal to $\nu_d(p)d^{n-2}(d-1)$.
\end{itemize}
By the Riemann-Hurwitz formula, we have
\[g_{n,p}^j=1+\dfrac{1}{2}\nu_d(p)d^{n-2}(pd-p-1-d)+\dfrac{1}{2}(M_{n,p}+K_{n,p})-\frac{1}{2}d^{n-2}(d-1)\sum_{k|p,k<p}k\nu_d(k)\varphi(p/k).\]
Here is a genus computation of some examples.
\begin{center}
\begin{tabular}{|c|c|c|c|c|c|c|}
\hline
\ \ $d$\ \ &\ \ $n$\ \ &\ \ $p$\ \ &\ $\nu_d(p)$\ &\ $M_{n,p}$\ &\ $K_{n,p}$\ &\ $g_{n,p}$\ \\\hline
3&1&1&3&0&0&0 \\\hline
3&2&1&3&4&2& 1\\\hline
2&2&2&2&2&0& 0\\\hline
2&3&2&2&7&1& 1\\\hline
2&2&3&6&6&0&2\\\hline
\end{tabular}
\end{center}

\begin{corollary}
Fix $n,p\geq1$, the surfaces $\VVV_{n,p}^j, 1\leq j\leq d-1$ are pairwise homeomorphic.
\end{corollary}
\begin{proof}
Topologically, the surface $\VVV_{n,p}^j$ is completely determined by the genus  and the number of ideal boundary points of $\widehat{\VVV}_{n,p}^j$, whereas these two numbers are independent of $j$.
\end{proof}

\begin{lemma}\label{compactification}
All ideal boundary points are critical points of $\widehat{\pi}_{n,p}^j$ with multiplicity $d-1$.
\end{lemma}
\begin{proof}
 To prove this lemma, we first give a symbolic description of the dynamics on the filled-in Julia set for a parameter outside the Multibrot set.

If $c\in \C\smm M_{d}$, the Julia set of $f_c$ is a Cantor set. If $c\in R_{M_{d}}(\theta)$ with $\theta\neq 0$ not necessarily periodic,  the dynamical rays $R_c(\theta/d)\ldots R_c\bigl((\theta+d-1)/d\bigr)$ bifurcate on the critical point. The set $R_c(\theta/d)\cup \ldots \cup R_c\bigl((\theta+d-1)/d\bigr)\cup \{0\}$ decomposes the complex plane into $d$ connected components. We denote by $U_0$ the component containing the dynamical ray $R_c(0)$ and by $U_1,\ldots,U_{d-1}$ the other components in counterclockwise order.

The orbit of a point $x\in K_c$ has an itinerary with respect to this partition. In other words, to each $x\in K_c$, we can associate a sequence in $ \Z_d^{\scriptsize\N}$ whose $j$-th entry is equal to $k$ if $f_c^{\circ j-1}(x)\in U_k$ . This gives a map $\iota_c:K_c\to \Z_d^{\scriptsize\N}$, which is bijective for any $c\in\C\setminus M_d$. Moreover, the dynamic of $f_c$ on $K_c$ is conjugate to shift on $\Z_d^{\scriptsize\N}$ via the map $\iota_c$.

Now let $\pi:=\pi_{n,p}^j\big|_{\EEE_{n,p,i}^j}$. The map $\pi:\EEE_{n,p,i}^j\longrightarrow\C\setminus M_d$ is a covering of degree $d_{n,p,i}^j$. Fix $c_0\in\C\setminus (M_d\cup R_{M_d}(0)),\ d_{n,p,i}^j=\#(\pi^{-1}(c_0))$. Since $\EEE_{n,p,i}^j$ is connected, the monodromy group induced by $\pi$, denoted by Mon$\left(\pi\right)$, acts on $\pi^{-1}(c_0)$ transitively. Then fixing any point $(c_0,z_0)\in \pi^{-1}(c_0)$,  the set $\pi^{-1}(c_0)$ is exactly the orbit of $(c_0,z_0)$ under Mon$\left(\pi\right)$.

Let $c_t:[0,1]\rightarrow\C\setminus M_d$ be a oriented simple closed curve based at $c_0$ such that $c_t$ intersects  $R_{M_d}(0)$ at only one point $c_{t_0}$.
Let $z_t$ be the $(n,p)$-preperiodic point of  $f_{c_t}$ obtained from the analytic continuation of $z_0$ along $c_t$. Note that as $c$ varies in $\C\setminus (M_d\cup R_{M_d}(0))$,
the $(n,p)$-preperiodic points of $f_c$, the dynamical rays $R_c(0)$ and $R_c\bigl((\theta_c+s)/d\bigr)\ (s\in\Z_d )$ move continuously. Consequently, we have
\[\left\{
\begin{array}{ccc}
\iota_{c_t}(z_t)&=&\iota_{c_0}(z_0)\qquad\text{for }t\in\left[0,\right.\left.t_0\right)\\
\iota_{c_t}(z_t)&=&\iota_{c_0}(z_1)\qquad\text{for }t\in\left(t_0,\right.\left.1\right]
\end{array}
\right.\]
Furthermore, on one hand, $z_t$ and $R_{c_t}(0)$ move continuously for $t\in[0,1]$. On the other hand, when $c_t$
passes through $R_{M_d}(0)$, the dynamical rays $R_{c_t}\bigl((\theta_t+s)/d\bigr)\ (s\in\Z_d)$ move discontinuously and jump from $R_{c_{t_-}}\bigl((\theta_{t_-}+s)/d\bigr)$ to $R_{c_{t_+}}\bigl((\theta_{t_+}+s+1)/d\bigr),\ t_{-}<t_0<t_+$. So if $\iota_{c_0}(z_0)=\beta_n\ldots\beta_1\overline{\epsilon_1\ldots\epsilon_p}$, then
\begin{equation}\label{1111}
\iota_{c_0}(z_1)=(\beta_n+1)\ldots(\beta_1+1)\overline{(\epsilon_1+1)\ldots(\epsilon_p+1)}
\end{equation}
Hence the map $\sigma_{c_t}$, an element of Mon$(\pi)$ induced by $c_t$, maps $(c_0,z_0)$ to $(c_0,z_1)$ with $z_1$ satisfying (\ref{1111}). Since $\pi_1(\C\setminus M_d,c_0)=\langle c_t\rangle$, then we have
\[ \left(\pi_{n,p}^j\big|_{\EEE_{n,p,i}^j}\right)^{-1}(c_0)=\bigl\{(c_0,z)\big|\iota_{c_0}(z)=(\beta_n+s)\ldots(\beta_1+s)\overline{(\epsilon_1+s)\ldots(\epsilon_p+s)},\ s\in\Z_d\bigr\}\]
As a consequence, $d_{n,p,i}^j=d$.
\end{proof}

\section{The Galois group of $Q_{n,p}(c,z)$}

The objective here is to study $\mathcal{X}_{n,p}$ from the algebraic point of view by calculating  its associated Galois group.

 Recall that  $\mathbf{C}$ denotes the ring $\C[c]$ and $\mathbf{K}$ is a fixed algebraically closed field containing $\mathbf{C}$.
Since the characteristic of $\C(c)$ is $0$, any polynomial  $f\in\mathbf{C}[z]$ induces a finite Galois extension $\C(c)(f)$ over $\C(c)$ (see \cite[Thm. 3.2.6, 2.7.14]{W}), where $\C(c)(f)$ is the splitting field of $f$, and hence a Galois group $G(f):=\on{Gal}(\C(c)(f)/\C(c))$.
 In particular, we denote the Galois group of $Q_{n,p}$ by $G_{n,p}$.

For each $n\geq0, p\geq1$, denote $\mathfrak{R}_{n,p}$ the set of roots of $Q_{n,p}\in\mathbf{C}[z]$. By (\ref{induction}), we have $f_c(\mathfrak{R}_{n,p})=\mathfrak{R}_{n-1,p}$ if $n\geq1$ and $f_c(\mathfrak{R}_{0,p})=\mathfrak{R}_{0,p}$. Let us consider
 $$\mathfrak{R}_{\leq n,p}:=\bigcup_{0\leq l\leq n}\mathfrak{R}_{l,p}.$$
Then $f_c(\mathfrak{R}_{\leq n,p})\subset \mathfrak{R}_{\leq n,p}$ and the action of $f_c$ induces  a directed graph structure consisting of a certain number of disjoint cycles of order $p$, on each vertex of which is attached a tree of height $n$. More precisely, for each $0\leq l\leq n$, we consider the roots in $\mathfrak{R}_{l,p}$ as the vertices of level $l$, and two vertices $\Delta_1,\Delta_2\in \mathfrak{R}_{\leq n,p}$ are connected by an oriented edge from $\Delta_1$ to $\Delta_2$ if $f_c(\Delta_1)=\Delta_2$. Thus $\mathfrak{R}_{\leq n,p}$ has a graph structure, and we denote this graph by $\mathfrak{R}_{\leq n,p}^T$ (see Figure 1).

{\noindent\bf Example.} For $d=3,p=4,n=2$, the directed graph $\mathfrak{R}_{\leq 2,4}^T$ has $18$ connected component, which are pairwise isomorphic. We draw one in the following.
\begin{figure}[htbp]\centering\label{pic1}
\includegraphics[width=14cm]{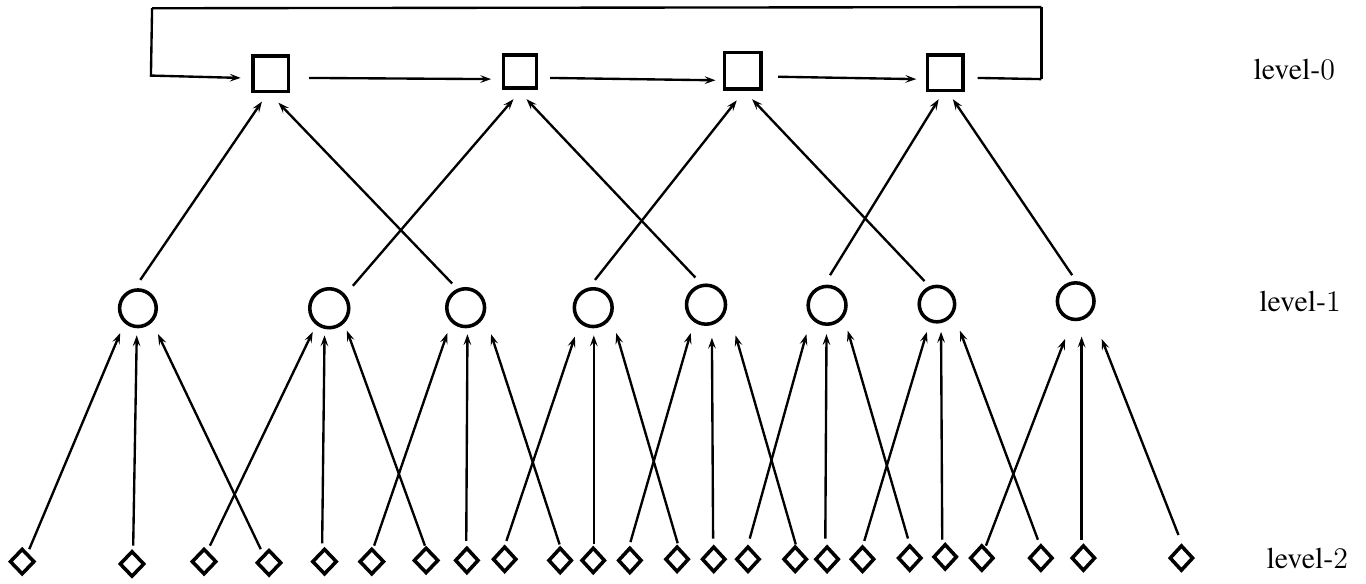}
\begin{center}{\sf Figure 1. A connected component of $\mathfrak{R}_{\leq 2,4}^T$.}
\end{center}
\end{figure}

On the other hand, note that $\mathfrak{R}_{\leq n,p}$ is the set of roots of
$$Q_{\leq n,p}:=\prod_{l=0}^n Q_{l,p}\in\mathbf{C}[z].$$
So, correspondingly, we consider the Galois group $G_{\leq n,p}$ of $Q_{\leq n,p}$.
Firstly, we have the following simple result.
\begin{proposition}\label{equal}
For each $n\geq 0,p\geq 1$, we have $G_{n,p}=G_{\leq n,p}$. 
\end{proposition}
\begin{proof}
By (\ref{induction}), any root of $Q_{l,p}\in \mathbf{C}[z]$ ($0\leq l\leq n$) can be written as a polynomial  with coefficients in $\mathbf{C}$ of roots of $Q_{n,p}$. It follows that the splitting field of $Q_{\leq n,p}=\Pi_{l=0}^nQ_{l,p}$ over $\C(c)$ is the same as that of $Q_{n,p}$ over $\C(c)$. Hence $G_{\leq n,p}=G_{n,p}$.
\end{proof}

By this proposition, computing the Galois group $G_{n,p}$ is equivalent to computing the group $G_{\leq n,p}$. Let $\sigma$ be an element in $G_{\leq n,p}$. Since it fixes the base field $\C(c)$ pointwise, we have  $\sigma(\mathfrak{R}_{l,p})=\mathfrak{R}_{l,p}$ and $f_c\circ\sigma=\sigma\circ f_c$.  Hence $\sigma$ induces an {\it automorphism} of the graph $\mathfrak{R}_{\leq n,p}^T$, i.e., $\sigma$ is a permutation on each $l$-level vertices of $\mathfrak{R}_{\leq n,p}^T$, and $\Delta_1,\Delta_2\in \mathfrak{R}_{\leq n,p}$ are connected by an edge from $\Delta_1$ to $\Delta_2$ if and only if $\sigma(\Delta_1),\sigma(\Delta_2)$ are connected by an edge from $\sigma(\Delta_1)$ to $\sigma(\Delta_2)$. Clearly, different elements of $G_{\leq n,p}$ induce different automorphisms of $\mathfrak{R}_{\leq n,p}^T$. So $G_{\leq n,p}$ can be seen as a subgroup of $\on{Aut}(\mathfrak{R}_{\leq n,p}^T)$, the automorphic group of the graph $\mathfrak{R}_{\leq n,p}^T$.

In  the case  $d=2$, Bousch \cite{B} proved that $$G_{\leq n,p}\simeq \on{Aut}(\mathfrak{R}_{\leq n,p}^T)\simeq H_{\leq n,p}(f_c),$$ where $H_{\leq n,p}(f_c)$ denotes the set of all permutations on $\mathfrak{R}_{\leq n,p}$ that commute with $f_c$.
In the general case, the result is similar but needs  a small modification. We will exhibit this point in the following.

Let $\sigma$ be an element of the Galois group $G_{\leq n,p}$.  As $\sigma$ fixes the field $\C(c)$ pointwise,  it must satisfy the following two conditions:
\begin{description}
\item [(P1)]  $\sigma$ commutes with $f_c$, \quad i.e., $\sigma\circ f_c=f_c\circ \sigma$.
\item [(P2)] $\sigma$ commutes with the rotation of argument $1/d$. That is,  if $\sigma(\Delta)=\widetilde{\Delta}$ for $\Delta,\widetilde{\Delta}\in \mathfrak{R}_{\leq n,p}$, then
$\sigma(\omega^j\Delta)=\omega^j\widetilde{\Delta}, \text{where }\omega=e^{\frac{2\pi i}{d}}\text{ and }1\leq j\leq d-1$

\end{description}
Therefore, if a permutation on $\mathfrak{R}_{\leq n,p}$ wants to be a candidate of elements in the Galois group $G_{\leq n,p}$, it should satisfy the conditions (P1) and (P2).

In fact, in the case of $d=2$, the condition (P1) implies (P2). To see this, let $\Delta_{n-1}$ be a root of $Q_{n-1,p}$ ($n\geq1$) and   $\Delta_n,-\Delta_n$ be the preimages of $\Delta_{n-1}$ under $f_c$. Let $\sigma$ be an element of $G_{\leq n,p}$ and  set $\widetilde{\Delta}_n:=\sigma(\Delta_n)$. By condition (P1) $\sigma$ must map $-\Delta_n$ to $-\widetilde{\Delta}_n$, then the condition (P2) holds. Therefore, it is possible for (P1)  to be a sufficient condition for a permutation on $\mathfrak{R}_{\leq n,p}$ to be an element of $G_{\leq n,p}$, and Bousch \cite{B} proved this point.

However, the situation has a little difference in the case of $d\geq3$.  Following the notations $\Delta_{n-1}, \Delta_n, \widetilde{\Delta}_n$ and $\sigma$ as above. In this case, $\Delta_{n-1}$  has at least $3$ preimages, which are $\Delta_n,\omega\Delta_n,\ldots,\omega^{d-1}\Delta_n$.  By condition (P1), we only know that $\sigma$ maps $\{\omega\Delta_n,\ldots,\omega^{d-1}\Delta_n\}$ bijectively to $\{\omega\widetilde{\Delta}_n,\ldots,\omega^{d-1}\widetilde{\Delta}_n\}$, but can not get $\sigma(\omega^j\Delta_n)=\omega^j(\widetilde{\Delta}_n)$ for $1\leq j\leq d-1$. So, in case of $d\geq3$, the condition (P2) can not be omitted.

What we would like to prove is that, except (P1) and (P2), no other restrictions are imposed on $G_{\leq n,p}$. The proof is  similar to that of Theorem $4$ in Chapter III of $\cite{B}$.
\begin{theorem}\label{Galois}
The Galois group $G_{\leq n,p}$ consists of all permutations on $\mathfrak{R}_{\leq n,p}$ which commute with $f_c$ and the rotation of argument $1/d$.
\end{theorem}
\begin{proof}
We denote $r_d$ the rotation of argument $1/d$, and $H_{\leq n,p}(f_c,r_d)$ the set of permutations on $\mathfrak{R}_{\leq n,p}$  which commute with $f_c$ and $r_d$. By the definition, it is not difficult to check that  $H_{\leq n,p}(f_c,r_d)$ leaves  each $\mathfrak{R}_{l,p}$, and hence $\mathfrak{R}_{\leq l,p}$ invariant for $0\leq l\leq n$.

Define a group homomorphism $$\phi_n:G_{\leq n,p}\to H_{\leq n,p}(f_c,r_d)$$ such that $\phi_n(\sigma)$ is the restriction of $\sigma$ to $\mathfrak{R}_{\leq n,p}$. According to the discussion above, we just need to prove the surjectivity of $\phi_n$.

Note first that the result is true for $n=0$ following 6 of Lemma \ref{period}.

 For $n=1$, since $H_{\leq 1,p}(f_c,r_d)$ leaves $\mathfrak{R}_{0,p}$ invariant, there is a natural homomorphism from $H_{\leq 1,p}(f_c,r_d)$ to $H_{\leq 0,p}(f_c,r_d)$ with $\widetilde{\tau}\mapsto \widetilde{\tau}|_{\mathfrak{R}_{0,p}}$. This homomorphism has an inversion which maps $\tau\in H_{\leq 0,p}(f_c,r_d)$ to $\widetilde{\tau}\in H_{\leq 1,p}(f_c,r_d)$ such that $\widetilde{\tau}|_{\mathfrak{R}_{0,p}}=\tau$ and $\widetilde{\tau}(\omega^j\Delta)=\omega^j\tau(\Delta)$ for each $\Delta\in \mathfrak{R}_{0,p},j\in[1,d-1]$. Thus $H_{\leq 1,p}(f_c,r_d)\simeq H_{\leq 0,p}(f_c,r_d)$. Note that
 $G_{1,p}=G_{0,p}$ (because the splitting fields of $Q_{0,p}$ and $Q_{1,p}$ over $\C(c)$ coincide), then the result is true for $n=1$.

 Now we argue by induction on $n$. Assume that $\phi_{n-1}:G_{\leq n-1,p}\to H_{\leq n-1,p}(f_c,r_d)$ is surjective ($n\geq2$).

Let $\tau\in H_{\leq n,p}(f_c,r_d)$. As $\tau$ commutes with $f_c$, it leaves $\mathfrak{R}_{\leq n-1,p}$ invariant. Then $\tau|_{n-1}$, the restriction of $\tau$ on $\mathfrak{R}_{\leq n-1,p}$, belongs to $ H_{\leq n-1,p}(f_c,r_d)$. By the assumption of induction, there is an element $\sigma_{n-1}$ of $G_{\leq n-1,p}$ with $\phi_{n-1}(\sigma_{n-1})=\tau|_{n-1}$. According to the Galois theory (see \cite[Thm. 2.88]{W}), we can find an element $\sigma$ of $G_{\leq n,p}$ such that its restriction on the splitting field of $Q_{\leq n-1,p}$ over $\C(c)$ coincides with $\sigma_{n-1}$. Set $\tau':=\tau\cdot \phi_n(\sigma)^{-1}$, then $\tau'\in H_{\leq n,p}(f_c,r_d)$ and  it fixes $\mathfrak{R}_{\leq n-1,p}$ pointwise. Now it remains to prove that $G_{\leq n,p}$ contains $\tau'$, i.e., there exists $\sigma'\in G_{\leq n,p}$ with $\phi_n(\sigma')=\tau'$, because if so, $\tau=\phi_n(\sigma')\phi_n(\sigma)=\phi_n(\sigma'\sigma)$, which implies $\phi_n$ is surjective.

Set $\kappa_l:=\nu_d(p)(d-1)d^{l-1}$ for each $l\geq 1$ (which is the number of roots of $Q_{l,p}$), and denote $$\mathfrak{R}_{n,p}=\{\Delta_n^i,\omega \Delta_n^i,\ldots,\omega^{d-1}\Delta_n^i;\}_{i=1}^{\kappa_{n-1}}$$
Since $\tau'$ fixes $\mathfrak{R}_{\leq n-1,p}$ pointwise and commutes with both $f_c,r_d$,  it can be expressed as a product
\[\tau'=\prod_{i=1}^{\kappa_{n-1}}\bigl(s_i,s_i+1,\cdots,d-1,0,\cdots,s_i-1\bigr),\]
where $(s_i,s_i+1,\cdots,d-1,0,\cdots,s_i-1)$ is the cyclic permutation on $(\Delta_n^i,\ldots,\omega^{d-1}\Delta_n^i)$ mapping $\Delta_n^i$ to $\omega^{{s_i}}\Delta_n^i$ and so on. Notice that all these cyclic permutations $(s_i,\ldots,s_i-1)$ are commutable.

The argument in this paragraph is a classical correspondence between the Galois theory and the covering theory; see \cite{Z} for the detail.
Let $V_{n,p}$ be the set of singular values of the projection $\pi: \mathcal{X}_{n,p}\to \C$.
Then $\pi_{n,p}$ restricts to a cover from the complement of the preimage of $V_{n,p}$ in $\mathcal{X}_{n,p}$ to the complement of $V_{n,p}$ in $\C$.
For all $c_0$ not in $V_{n,p}$, there is thus an action of $\pi_1(\C\setminus V_{n,p},c_0)$ on the roots
 $$Z_{n,p}=\{z_n^i,\ldots,\omega^{d-1}z_n^i\}_{i=1}^{\kappa_{n-1}}$$
 of $Q_{n,p}(c_0,z)$, which is seen as a polynomial in the variable $z$ and with complex coefficients. By the correspondence between  the Galois theory and the covering theory (see \cite[Thm. 1]{Z}),
there is a (non-unique) choice of bijection between the roots of $Q_{n,p}\in\mathbf{C}[z]$ and the roots of $Q_{n,p}(c_0,z)\in \C[z]$ such that the set of permutations induced by the Galois group on the set  $\mathfrak{R}_{\leq n,p}$ is conjugated by this bijection to the set of permutations on $Z_{n,p}$ induced by $\pi_1(\C\setminus V_{n,p},c_0)$. Thus we get a surjective (not injective, usually) morphism from $\pi_1(\C\setminus V_{n,p},c_0)$ to the Galois group.
Moreover, this bijection is such that any polynomial relation between the $\Delta_n^i$ with coefficient in $\C(c)$ will give a relation between the $z_n^i$, with $c_0$ being substituted to $c$. It implies that the action of $\pi_1(\C\setminus V_{n,p},c_0)$ on $Z_{n,p}$ preserves commutation with multiplication by $\omega$.

Therefore, by the discussion above,
to obtain the required permutation $\tau'$, we only need to find a path  in the basic group  $\pi_1(\C\setminus V_{n,p},c_0)$  such that its monodromy action on
$\{(z_n^i,\ldots,\omega^{d-1}z_n^i)\}_{i=1}^{\kappa_{n-1}}$ induces the same permutation as $\tau'$.  We now show the following result, which is sufficient: for any $i\in[1,\kappa_{n-1}]$, there exists a path in  $\pi_1(\C\setminus V_{n,p},c_0)$ such that its monodromy action induces the permutation $(s_i,s_i+1,\cdots,s_i-1)$.

Fix any $i\in[1,\kappa_{n-1}]$. Suppose that $\{(c_0,z_n^i),\ (c_0,\omega z_n^i),\ldots,(c_0,\omega^{d-1}z_n^i) \}$ belong to $\VVV_{n,p}^{t}$. Let $\hat{c}$ be an $(n,p)$-Misiurewicz parameter with $(\hat{c},0)\in \VVV_{n,p}^t$.  Such $\hat{c}$ exists because the set $D_{n,p}^t(\text{Misiurewicz})$ is non-empty (see section 5.2 item $2$). By (\ref{induction}), $(\hat{c},\hat{c})\in\mathcal{X}_{n-1,p}$. Since $\hat{c}$ is a Misiurewicz parameter and the orbit of $\hat{c}$ does not contain $0$, then the point $(\hat{c},\hat{c})$ belongs to no sets in Lemma \ref{critical} in the case $l=n-1$. Hence $w=\hat{c}$ is a simple root of the equation $Q_{n-1,p}(\hat{c},w)=0$ (in $w$). So by the Implicit Function Theorem, the equation $Q_{n-1,p}(c,w)=0$ admits a unique solution $w=w(c)$ close to $\hat{c}$ fullfilling $w(\hat{c})=\hat{c}$. Thus, a neighborhood of $(\hat{c},0)$ in $\mathcal{X}_{n,p}$ can be written as
\[\bigl\{(c,z_c)\cup(c,\omega z_c)\cup\cdots\cup(c,\omega^{d-1}z_c)\mid |c-\hat{c}|<\epsilon\bigr\}\]
where $z_c$ is one of the preimages of $w(c)$ under $f_c$ nearby $0$.

When $c$ makes a small turn around $\hat{c}$, the set $\{z_c,\ldots, \omega^{d-1}z_c\}$ gets a cyclic permutation with $\omega^jz_c$ mapped to $\omega^{j+1}z_c$, because $\pi_{n,p}$ is a degree $d$ covering in a punctured neighborhood of $(\hat{c},0)$ (which is shown in \S 5.2 item $2$), and the other $(n,p)$-preperiodic points of $f_c$  remain fixed,  since $\pi_{n,p}$ is injective around each point $(\hat{c},\xi)$ with $\xi$ a non-zero $(n,p)$-preperiodic point of $f_{\hat{c}}$. So if we choose a path $\gamma\in\pi_1(\C\setminus V_{n,p},c_0)$ homotopic to  $\hat{c}$,  the induced permutation by its monodromy action gives the cyclic permutation $(2,\ldots,d,1)$ on
 $(z_n^*,\ldots,\omega^{d-1}z_n^*)$ for a $(n,p)$-preperiodic point $z_n^*$ of $f_{c_0}$ fullfilling $(c_0,z_n^*)\in \VVV_{n,p}^{t}$, and keeps the other $(n,p)$-preperiodic points of $f_{c_0}$ fixed. Now we join $(c_0,z_n^i)$ and $(c_0,z_n^*)$ by a curve from $(c_0,z_n^i)$ to $(c_0,z_n^*)$ in $\VVV_{n,p}^{t}\setminus\pi_{n,p}^{-1}(V_{n,p})$, and denote  its projection under $\pi_{n,p}$ by $\beta$. Then $\beta\in\pi_1(\C\setminus V_{n,p},c_0)$ and the path $\beta\cdot\gamma^{s_i}\cdot\beta^{-1}$ is what we expect.
\end{proof}

Applying this theorem, we can also characterize the Galois group $G_{\leq n,p}$ by the automorphisms of the directed graph $\mathfrak{R}_{\leq n,p}^T$, as in the $d=2$ case. For $d\geq 3$, denote by
$\on{Aut}(\mathfrak{R}_{n,p}^T,r_d)$ the set of automorphisms of $\mathfrak{R}_{\leq n,p}^T$ that commute with the rotation of argument $1/d$, and by  $H_{\leq n,p}(f_c,r_d)$ the set of permutations  on $\mathfrak{R}_{\leq n,p}$  that commute with $f_c$ and the rotation of argument $1/d$.
\begin{corollary}
For $n\geq 0,p\geq1$, $G_{\leq n,p}\simeq \on{Aut}(\mathfrak{R}_{n,p}^T,r_d)\simeq H_{\leq n,p}(f_c,r_d)$.
\end{corollary}

Following Bousch \cite[Chap. 3, III]{B} and Silverman (\cite[$\S$ 3.9]{Sil}), we express the Galois group $G_{n,p}$ ($n\geq2$) as a wreath product.
\begin{definition}
Let $G$ be a group and $\Sigma$ be a subgroup of $\mathbf{S}_m$, where $\mathbf{S}_m$ denotes the set of permutations on $\{1,\ldots,m\}$. Denote by $\Sigma\ltimes G^m$ the wreath product  of $G$ and $\Sigma$. As a set, it consists of $g=\sigma(g_1,\cdots,g_m)$ where $g_i\in G$ and $\sigma\in \Sigma$. The multiplication is defined by
\[g\cdot h=\sigma_g(g_1,\cdots,g_m)\cdot\sigma_h(h_1,\cdots,h_m)=\sigma_g\circ\sigma_h(g_{\sigma_h(1)}\cdot h_1,\cdots,g_{\sigma_h(m)}\cdot h_m).\]
Under this multiplication law, $ \Sigma\ltimes G^m$ is a group with $g^{-1}=\sigma_g^{-1}\bigl(g_{\sigma_g^{-1}(1)}^{-1},\cdots,g_{\sigma_g^{-1}(1)}^{-1}\bigr)$ and unit element
$(1,\ldots,1)$.
\end{definition}

In \cite{B}, Bousch showed that the Galois group $G_{0,p}$ is isomorphic to the wreath product $ \mathbf{S}_{\nu_d(p)/p}\ltimes (\Z/p\Z)^{\nu_d(p)/p}$ (see also \cite[$\S$ 3.9]{Sil}).
From the proof of Theorem \ref{Galois}, we have seen that $G_{1,p}=G_{0,p}$, so $$G_{1,p}\simeq \mathbf{S}_{\nu_d(p)/p}\ltimes (\Z/p\Z)^{\nu_d(p)/p}.$$ For $n\geq 2$, we can give inductively an isomorphic model of $G_{n,p}$ by a wreath product. Recall that $\kappa_n=\nu_d(p)(d-1)d^{n-1}$ ($n\geq 1$) is the number of roots of $Q_{n,p}$.
\begin{proposition}\label{wreath}
For $n\geq2$, we have $G_{n,p}\cong G_{n-1,p}\ltimes (\Z/d\Z)^{\kappa_{n-1}}$,  where the action of $G_{n-1,p}$ on $(1,2,..., \kappa_{n-1})$ comes from the action of $G_{n-1,p}$ on the roots of $Q_{n-1,p}$, of which there are exactly $\kappa_{n-1}$.

\end{proposition}
\begin{proof}
 For $n\geq2$, we denote $(\Delta_{n-1}^i)_{i=1}^{\kappa_{n-1}}$ the roots of $Q_{n-1,p}\in\mathbf{C}[z]$, and denote $$(\ \{\Delta_n^i,\ldots,\omega^{d-1}\Delta_n^i\}\ )_{i=1}^{\kappa_{n-1}}$$
the roots of $Q_{n,p}$ such that $f_c(\Delta_n^i)=\Delta_{n-1}^i$.
\begin{center}
\begin{tabular}{c|c|c|c}
$\Delta_{n-1}^1$&$\Delta_{n-1}^2$&$\cdots\cdots$&$\Delta^{\kappa_{n-1}}_{n-1}$\\[5pt]
\hline
$\Delta^1_n$&$\Delta_n^2$&$\cdots\cdots$&$\Delta^{\kappa_{n-1}}_n$\\
$\omega \Delta^1_n$&$\omega \Delta_n^2$&$\cdots\cdots$&$\omega \Delta^{\kappa_{n-1}}_n$\\
$\vdots$&$\vdots$& &$\vdots$\\
$\omega^{d-1} \Delta^1_n$&$\omega^{d-1}\Delta_n^2$&$\cdots\cdots$&$\omega^{d-1} \Delta^{\kappa_{n-1}}_n$
\end{tabular}
\end{center}
We define a group homomorphism
$$W:G_{n,p}\longrightarrow G_{n-1,p}\ltimes(\Z/d\Z)^{\kappa_{n-1}}$$ such that $W(\sigma)=\sigma|_{n-1}(s_1,\ldots,s_{\kappa_{n-1}})$, where $\sigma|_{n-1}$ is the restriction of $\sigma$ on the splitting field of $Q_{n-1,p}$ over $\C(c)$, and the $i$-th digit in $(s_1,\ldots,s_{\kappa_{n-1}})$ is $s_i$ if and only if  once $\sigma(\Delta_{n-1}^i)=\Delta_{n-1}^t$ for some $1\leq t\leq \kappa_{n-1}$, then $\sigma(\Delta_n^i)=\omega^{s_i}\Delta_n^{t}$.
The injectivity of $W$ is straightforward by the action of $G_{n,p}$ on $\mathfrak{R}_{\leq n,p}$ and the subjectivity of $W$ is due to Theorem \ref{Galois}.
\end{proof}

Before ending this section, we give a computation of $G_{n,p}$ for some small $n,p$. Note that although $G_{1,p}$ is isomorphic to a subgroup $\mathbf{S}_{\nu_d(p)/p}\ltimes (\Z/p\Z)^{\nu_d(p)/p}$ of $\mathbf{S}_{\nu_d(p)}$, it is indeed a subgroup of $\mathbf{S}_{\nu_d(p)(d-1)}$. So mimicking the action of $G_{1,p}$ on $$\{\omega\Delta_1^1,\ldots,\omega\Delta_1^{\nu_d(p)};\ \ldots;\ \omega^{d-1}\Delta_1^1,\ldots,\omega^{d-1}\Delta_1^{\nu_d(p)}\},$$
we define a subgroup $\mathbf{P}_{\nu_d(p)(d-1),d}$ of $\mathbf{S}_{\nu_d(p)(d-1)}$ such that $\tau\in \mathbf{P}_{\nu_d(p)(d-1),d}$ if and only if
\begin{eqnarray*}
&&\tau\big(1,\ldots,\nu_d(p);\ \ldots\ ;(d-2)\nu_d(p)+1,\ldots,(d-2)\nu_d(p)+\nu_d(p)\big)\\
&=&\big(\sigma(1),\ldots,\sigma(\nu_d(p));\ \ldots\ ;(d-2)\nu_d(p)+\sigma(1),\ldots,(d-2)\nu_d(p)+\sigma(\nu_d(p))\big)
\end{eqnarray*}
for a  $\sigma\in \mathbf{S}_{\nu_d(p)}$. Then $\mathbf{P}_{\nu_d(p)(d-1),d}\simeq \mathbf{S}_{\nu_d(p)}\simeq G_{1,p}$, and $\mathbf{P}_{\nu_d(p)(d-1),d}= \mathbf{S}_{\nu_d(p)}$ in  the case $d=2$. The results of computation are listed in the following table.
\begin{center}
\begin{tabular}{|c|c|c|c|c|c|}
\hline
\ \ $d$\ \ &\ \ $n$\ \ &\ \ $p$\ \ &\ $\nu_d(p)$\ &\ $\kappa_{n-1}$\ &\ $G_{n,p}(d)$\ \\\hline
3&1&1&3&--&$\mathbf{S}_{3}\simeq\mathbf{P}_{6,3}$ \\\hline
3&2&1&3&6&$\mathbf{P}_{6,3}\ltimes(\Z/3\Z)^6$\\\hline
3&3&1&3&18&($\mathbf{P}_{6,3}\ltimes(\Z/3\Z)^6)\ltimes(\Z/3\Z)^{18}$\\\hline
2&3&2&2&4&$((\Z/2\Z)\ltimes(\Z/2\Z)^2)\ltimes(\Z/2\Z)^4$\\\hline
2&2&3&6&6&$((\Z/2\Z)\ltimes(\Z/3\Z)^2)\ltimes(\Z/2\Z)^6$\\\hline
\end{tabular}
\end{center}

\end{document}